\documentclass[a4,reqno, final]{amsart}
\usepackage[english]{babel}

\usepackage[usenames,dvipsnames]{xcolor}

\usepackage{enumerate}
\usepackage{appendix}
\usepackage[colorinlistoftodos]{todonotes}
\usepackage{soul}

\usepackage[colorlinks=false,backref=false,pagebackref=false,pdftex,pdfauthor={Bonheure, F\" oldes, Moreira dos Santos, Salda\~na, Tavares},pdftitle={}]{hyperref}
\usepackage{amsmath,amsthm,amsfonts,latexsym,amssymb}
\usepackage{comment}
\usepackage[showonlyrefs]{mathtools}
\mathtoolsset{showonlyrefs=true}

\usepackage[notcite,notref]{showkeys}

\newtheorem*{namedtheorem}{\theoremname}
  \newcommand{\theoremname}{testing}

\begin{document}

\title[Paths to Uniqueness]
{Paths to uniqueness of critical points and applications to partial differential equations}

\author[Bonheure]{Denis Bonheure}
\address{Denis Bonheure, Juraj F\"oldes and  Alberto Salda\~na\newline \indent D{\'e}partement de Math{\'e}matique --- Universit{\'e} libre de Bruxelles \newline \indent CP 214,  Boulevard du Triomphe, B-1050 Bruxelles, Belgium}
\email{denis.bonheure@ulb.ac.be, \ Juraj.Foldes@ulb.ac.be, \ asaldana@ulb.ac.be}

\author[F\" oldes]{Juraj F\"oldes}

\author[Moreira dos Santos]{Ederson Moreira dos Santos}
\address{Ederson Moreira dos Santos \newline \indent Instituto de Ci{\^e}ncias Matem{\'a}ticas e de Computa\c{c}{\~a}o --- Universidade de S{\~a}o Paulo \newline \indent
Caixa Postal 668, CEP 13560-970 - S\~ao Carlos - SP - Brazil}
\email{ederson@icmc.usp.br}

\author[Salda\~na]{Alberto Salda\~na}

\author[Tavares]{Hugo Tavares}
\address{Hugo Tavares \newline \indent Center for Mathematical Analysis, Geometry and Dynamical Systems,
 \newline \indent Mathematics Department --- Instituto Superior T\'ecnico, Universidade de Lisboa   \newline \indent
Av. Rovisco Pais, 1049-001 Lisboa, Portugal}
\email{htavares@math.ist.utl.pt}

\date{\today}

\subjclass[2010]{46N10; 49K20; 35J10; 35J15; 35J47; 35J62}

\keywords{Uniqueness of critical points; Non-convex functionals; Uniqueness of positive solutions to elliptic equations and systems.
}

\begin{abstract}
We prove a unified and general criterion for the uniqueness of critical points of a functional in the presence of 
constraints such as positivity, boundedness, or fixed mass.  Our method relies on convexity properties along suitable paths and 
significantly generalizes well-known uniqueness theorems. Due to the flexibility in the construction of the paths, 
our approach does not depend on the convexity of the domain and can be used to prove uniqueness in subsets, even if it does not hold globally.
The results apply to all critical points and not only to minimizers, thus they provide uniqueness of solutions to the corresponding Euler-Lagrange equations.
For functionals emerging from elliptic problems, the assumptions of our abstract theorems follow from maximum principles, decay properties, and novel general inequalities.
To illustrate our method we present a unified proof of known results, as well as new theorems for mean-curvature type operators, fractional Laplacians, Hamiltonian systems, 
Schr\" odinger equations, and Gross-Pitaevski systems. 
\end{abstract}

\maketitle
\numberwithin{equation}{section}
\newtheorem{theorem}{Theorem}[section]
\newtheorem{lemma}[theorem]{Lemma}
\newtheorem{example}[theorem]{Example}
\newtheorem{remark}[theorem]{Remark}
\newtheorem{proposition}[theorem]{Proposition}
\newtheorem{definition}[theorem]{Definition}
\newtheorem{corollary}[theorem]{Corollary}
\newtheorem*{open}{Open problem}

\tableofcontents

\section{Introduction}

The existence of critical points of a functional $I$ traditionally follows from general
 arguments based on Direct or Minimax Methods.

However, the uniqueness
of critical points is in general a  subtle issue, depending both on local and global properties of a functional,  and  
essentially on its domain of definition. In particular, this domain might reflect conserved quantities (such as mass or energy), or one-sided constraints (e.g. positivity or boundedness). 

\smallbreak
 Probably, the best known uniqueness result in the calculus of variations is the following: If $I$ is a strictly convex functional, defined 
on an open convex subset of a normed space, then it has at most one critical point, which is the global minimum
whenever it exists. However, problems with constraints have frequently non-convex domains and even for convex ones, 
such as the cone of positive functions, the convexity of $I$ is a very restrictive requirement. For example, in many problems zero is a critical point of $I$ (or solution of the corresponding Euler-Lagrange equation), and
if there exists a positive critical point then $I$ cannot be strictly convex in the cone of positive functions.

\smallbreak
In this paper we provide a simple yet general condition guaranteeing the uniqueness of critical points, which relies 
on an elementary observation: If $I$ is a smooth functional and $\gamma$ is a smooth curve connecting two critical 
points of $I$, then $t \mapsto I(\gamma(t)) =: F(t) \in \mathbb{R}$ cannot be a strictly convex function.
Indeed, since $\gamma$ connects critical points, the derivative of $F$ at the endpoints must vanish, which is impossible for  
a strictly convex function. This observation yields a uniqueness result whenever an appropriate curve $\gamma$ can be found, and below we  
construct such $\gamma$ for many problems involving nonlinear partial differential equations (pdes).
These examples contain new uniqueness proofs and a shorter and unified approach to some known results.    

\smallbreak
We emphasize that requiring convexity along a particular curve is much weaker than assuming convexity of $I$. Moreover, such path-wise convexity can be defined on sets that are not necessarily convex. As such, it can be used as a fine tool to prove uniqueness of critical points with certain additional criteria. 
Our method can also be used to prove the simplicity of the eigenvalues of non-linear eigenvalue problems, where the uniqueness holds up to multiplication by scalars. 

\smallbreak
Recall that the uniqueness immediately implies that the critical point
inherits all symmetries of the problem. For example, if the functional and its domain are radially symmetric, then 
the critical point possesses the same symmetry. 
Furthermore, the uniqueness simplifies the dynamics of the gradient flow induced by the functional, and 
in many cases provides global stability properties of equilibria. 

\smallbreak
Our main uniqueness result is formulated in a general abstract framework to allow applications to various problems. 

\begin{theorem}\label{th:abstracttheorem}
Let $(X, \|\cdot\|)$ be a normed space, $I: X \rightarrow {\mathbb R}$  be a Fr\'echet differentiable functional, and $A \subset X$ be a 
subset of critical points of $I$. If for all $u, v \in A$ there exists a map $\gamma : [0,1] \to X$ such that
 \begin{enumerate}[a)]
 \item $\gamma(0) = u$, $\gamma(1)=v$,
 \item $\gamma$ is locally Lipschitz at $t=0$, that is, $\|\gamma(t) - u\| \leq C t$ 
  for each $t \in [0, \delta]$ and some $\delta > 0$,
 \item $t\mapsto I(\gamma(t))$ is convex in $[0, 1]$,
 \end{enumerate}
then
 \begin{enumerate}[i)]
 \item  $I$ is constant on $A$, 
 \item $t\mapsto I(\gamma(t))$ is a constant function, 
 \item if condition c) holds with strict convexity  for  all $u \neq v$, then $A$ has at most one element.
 \end{enumerate}
 \end{theorem}

We readily see  that every (strictly) convex functional satisfies conditions a)-c) of Theorem \ref{th:abstracttheorem} (resp. with strict inequality at c)) for $\gamma(t) = (1-t)u + tv$. 
The linear structure on $X$ is not essential and can be replaced by a differential structure, which is needed 
for the notion of critical point. Rather than introducing a new notation and restating the problem on Banach or Fr\' echet manifolds, 
we formulate a consequence of Theorem 
\ref{th:abstracttheorem} for constraint minimization problems with applications to elliptic problems in mind.

\begin{corollary}\label{cor:abstractcor}
Let $X,Y$ be Banach spaces,  $I: X \rightarrow {\mathbb R}$ and $R: X \rightarrow Y$ be Fr\'echet differentiable. Suppose that $0$ is a regular value of $R$, set $S := R^{-1}(0)$, and let $A \subset S$ be a subset of critical points of $I|_S$. If for all $u,v \in A$ 
there exists $\gamma : [0,1] \to S$ satisfying a)--c) of Theorem \ref{th:abstracttheorem}, then i)--iii) from Theorem \ref{th:abstracttheorem} hold
with $I$ replaced by $I|_S$.
\end{corollary}

We apply Theorem \ref{th:abstracttheorem} and Corollary \ref{cor:abstractcor} 
to several pde problems with variational structure. In these applications, assumption a) is easily fulfilled, the main challenge is the construction of paths $\gamma$ satisfying b) and c). The condition b) is in fact an assumption on the parametrization $\gamma$. Actually, if b) is not satisfied 
we cannot conclude that the derivative of $t \mapsto I(\gamma(t))$ vanishes at $t=0$ even if $ \nabla I(\gamma(0)) = 0$. This is not a technical obstacle, since the uniqueness of critical points does not hold  if we require H\" older continuity in b) instead of Lipschitz continuity.  
Indeed, consider the following one-dimensional example
\begin{equation}\label{eq:odex}
I(x) = \frac{1}{3}x^3 - x - \frac{2}{3} = \frac{1}{3}(x + 1)^3 - (x + 1)^2, \qquad \gamma(t) = \sqrt{2}\sqrt{1 + t} - 1.
\end{equation}
It is easy to verify that $I$ has two critical points $\pm 1$, $\gamma(\pm 1) = \pm 1$, and 
$t \mapsto I(\gamma(t)) =  \frac{2\sqrt{2}}{3} (1 + t)^{\frac{3}{2}} - 2(1 + t)$ is a strictly convex function; however, $\gamma$ is not Lipschitz at $t = -1$. 
The verification of b) is a subtle issue and in our examples, where $X$ is a function space and critical points are solutions of certain elliptic problems, it strongly relies on comparison properties of endpoints of $\gamma$, which follow either from the Hopf lemma 
or from sharp decay estimates at infinity for solutions of the Euler-Lagrange equation. 
In this step we strongly
use that the endpoints of $\gamma$
are critical points of $I$ and not arbitrary elements of $X$. 

\smallbreak
The assumption c) has a different flavor and it heavily depends on indirect convexity properties of $I$, which are manifested by 
 sharp and delicate inequalities for arbitrary endpoints $u, v$
(not necessarily critical points).
As noted in Theorem \ref{1.1'} below, the convexity in c) can be weakened to conclude i) of Theorem \ref{th:abstracttheorem}; however, in that case
we need additional assumptions to conclude the uniqueness of critical points. 

\smallbreak

Our main results also provide novel insights to problems where uniqueness cannot be established. For instance, 
if the set $A$ contains a global minimizer of $I$, it follows from i) and ii) that all critical points 
are global minimizers. Moreover, if $A$ contains at least two points, then they are not isolated since every point on $\gamma$ is a global minimizer 
(the continuity of $\gamma$ at $t=0$ follows from  b)). Another fine application, illustrated below for the 
mean curvature operator, is the uniqueness of small solutions, even if the existence of additional large solutions is known.

\smallbreak

The idea of proving uniqueness by using a generalized convexity assumptions is not entirely new in the literature. 
The closest to our results is  \cite{BelloniKawohl}, where the key idea there can be rephrased in our setting as: There is no curve $\gamma$ 
connecting two global minimizers of $I$ such that $t \mapsto I(\gamma (t))$ is strictly convex. This is indeed true, since 
strictly convex functions on an interval attain 
a strict maximum at an endpoint, a contradiction to 
global minimality of endpoints. Although this method is presumably applicable to curves connecting strict local minima, it fails for general
 critical points, see for instance our one-dimensional example \eqref{eq:odex}. There are many 
papers applying  the ideas of \cite{BelloniKawohl} in various settings, see for instance \cite{ABGS2011,BT2007, BrascoFranzina, DellaPietra2013,PG2014, FM2004, KN2008, LL2014, Lucia2006,TV2013}, however, none contains results as 
general as Theorem \ref{th:abstracttheorem} or Corollary \ref{cor:abstractcor}.

\smallbreak

The cornerstone of \cite{BelloniKawohl} (which deals with elliptic equations involving the $p$--Laplace operator),
and an important ingredient in many of our examples, is the inequality 
\[
|\nabla \gamma (t)(x))| \leq \Big((1-t) |\nabla u(x)|^p + t |\nabla v(x)|^p\Big)^\frac{1}{p} 
\qquad \textrm{for }  u, v > 0 \,,
\]
where $\gamma(t)(x) := ({(1-t) u^p(x) + t v^p(x)})^{\frac{1}{p}}$. This inequality can be traced back to \cite{BenguriaBrezisLieb, DiazSaa} and in 
our manuscript
we prove a more general version. 
Such path $\gamma$ was used  in many papers to show the uniqueness of positive solutions for elliptic equations in mathematical physics, for instance,
\cite[Lemma A.4]{LiebSeiringerYngvason} studies a Gross-Pitaevskii energy functional and \cite[Lemma 4]{BenguriaBrezisLieb} treats a Thomas-Fermi-von Weizs\"acker functional.  Frequently, only minimizers are considered due to physical considerations, and therefore the strict convexity of $t \mapsto I(\gamma(t))$ suffices to 
prove uniqueness. 

\smallbreak
Another approach to uniqueness relies on the existence of a strict variational sub-symmetry, see \cite{Reichel}. 
Specifically, if $u_0$ is a critical point of $I$ and 
there is a family of maps $(g_\epsilon)$ with a group structure such that 
$I(g_\epsilon(u)) < I(u)$ for $u \neq u_0$, then $u_0$ is the only critical point. 
In some particular settings the abstract method of \cite{Reichel} can be interpreted as an infinitesimal version of our approach. 

\smallbreak

To show an application of Theorem \ref{th:abstracttheorem}, let us consider the following basic model problem, which already contains 
most of the important ingredients. 
Let $\Omega \subset {\mathbb R}^N$ be a regular bounded domain, $1 <  q <2$, and define the functional
on the Sobolev space $H^1_0(\Omega)$ by
\begin{gather}
I(u) := \frac{1}{2} \int_{\Omega} | \nabla u|^2 \, dx  - \frac{1}{q} \int_{\Omega} |u|^{q} \, dx \,. 
\end{gather}
It is well known that $I$ has infinitely many critical points, see for example \cite[Theorem 1 (b)]{BartschWillem}; 
however, there is only one \emph{positive} critical point of $I$, see \cite{Krasnoselskii, KellerCohen, Amann1971, Hess1977, deFigueiredo1982, BrezisOswald}.
Although $I$ is not strictly convex on the cone of positive functions of $H^1_0(\Omega)$
, we can verify the
uniqueness of positive critical points using Theorem \ref{th:abstracttheorem} with 
\[
\gamma(t) = \sqrt{(1-t) u^2 + t v^2} \,.
\]
Indeed, as $q < 2$, one obtains that $t \mapsto |\gamma(t)|^q$ is strictly concave, and therefore the second term of $t \mapsto I(\gamma(t))$ is strictly convex. The 
first term is convex by the general Lemma \ref{lem:gin} proved below. Intuitively, in our examples we employ a strong convexity property of
 the principal term to improve convexity properties 
of the nonlinear part. Observe that the critical points of $I$ are weak solutions of the Euler-Lagrange equation
\[
-\Delta u = |u|^{q-2}u \ \ \text{in} \ \ \Omega, \quad u = 0 \ \ \text{on} \ \ \partial \Omega \,,
\]
and  therefore the required local Lipschitz property follows from the comparison 
of positive solutions $u$ and $v$, see Lemma \ref{lemma:gpg} below, and Theorem \ref{th:abstracttheorem} yields the uniqueness.

\smallbreak

Our abstract results apply in far more general settings and here 
we focus on models arising for example in physics, engineering, 
and geometry.
We present simplified and unified proofs of known results and  
we also present new uniqueness results. 
Our goal is to present the main ideas and complications
in a comprehensible manner and to show the methods in a broad range of problems, rather than treating 
the most general setting or finding optimal assumptions. Our examples include equations and systems with quasilinear or nonlocal differential operators 
on bounded and unbounded domains with various boundary conditions. 

\smallbreak
In the first example we study a family of quasilinear problems
\begin{equation}\label{eq:qlpl}
-\textrm{div} (h(|\nabla u|^{p}) |\nabla u|^{p - 2} \nabla u) = g(x, u), \qquad p > 1 \,,
\end{equation} 
and under general assumptions on $h$ and $g$ we show new uniqueness results for positive solutions.
If $h \equiv 1$ the uniqueness was already established in \cite{DiazSaa}. In this case,
the left hand side reduces to the well-known $p$-Laplacian operator 
$(-\Delta_p)$, which is a nonlinear counterpart of the Laplacian ($p = 2$). It is used to model phenomena strongly characterized by nonlinear diffusion and it finds application, for example, in elasticity theory to model dilatant ($p>2$) or pseudo-plastic ($1<p<2$) materials. We show uniqueness results which in particular include sublinear \cite{DiazSaa} and Allen-Cahn-type $p$-Laplacian problems, either with Dirichlet or nonlinear boundary conditions. For the latter see \cite{Berestycki1981, Berger1979} for the particular case of $p=2$.  

\smallbreak
For $p = 2$ and $h(z) = (1 \pm z^2)^{-\frac{1}{2}}$ we obtain 
a problem involving the mean curvature operator in Eucledian or Minkowski space
\begin{align}\label{mcu:equation}
 {\mathcal M}_{\pm} u :=-\operatorname{div} \left(\frac{\nabla u}{\sqrt{1 \pm |\nabla u|^2}}\right)=g(x,u) \,. 
\end{align}

\smallbreak
The operator ${\mathcal M}_{+}$ is important in geometry. We refer to \cite{Giusti1984} for classical results on minimal surfaces and to \cite{ObersnelOmari2010,ObersnelOmariRivetti2014} for more references on boundary value problems involving this operator. It also classically appears  in the study of capillarity surfaces, see \cite{Finn1986}, and was proposed as a prototype in models of reaction processes with saturating diffusion, see \cite{KurganovRosenau2006} and the references therein. The natural functional space to look for solutions of \eqref{mcu:equation} with  ${\mathcal M}_{+}$ is the space of functions of bounded variations. We anticipate that our results lead to uniqueness of regular solutions.

\smallbreak
The operator ${\mathcal M}_{-}$ appears in the Born-Infeld electrostatic theory to include the \emph{principle of finiteness} in Maxwell's equations; see \cite{avenia,Mawhin,BartnikSimon}. Solutions to \eqref{mcu:equation} must satisfy $|\nabla u| < 1$ and can be obtained by minimization of a functional in a suitable convex subset of the Sobolev space $W^{1,\infty}(\Omega)$. Because of this, we need to formulate an auxiliary problem using truncations, which rely on fine quantitative regularity estimates. Since the set of critical points of the transformed problem might be larger than the original one, 
we need to exploit the fact that Theorem \ref{th:abstracttheorem} also applies to proper subsets of critical points.

\smallbreak
Our abstract results can also be applied to obtain new uniqueness results for nonlocal equations such as 
\begin{align*}
 (-\Delta)^s u(x) := \lim_{\varepsilon\to 0}\int_{|x-y|\geq\varepsilon} \frac{u(x)-u(y)}{|x-y|^{N+2s}}\, dy=g(x,u(x)) \,,
\end{align*}
where $s\in(0,1)$. 
The operator $(-\Delta)^s $ is often called  the (integral) fractional Laplacian and it appears as an infinitesimal generator of a L\'{e}vy process. It finds applications, for instance, in water waves models, crystal dislocations, nonlocal phase transitions, finance, flame propagation; see \cite{bucur,Hitchhiker}. The fractional Laplacian is a \emph{nonlocal} operator, since it encodes diffusion with large-distance interactions, and this nonlocality plays an essential role in the definition of the variational setting and in the application of Theorem \ref{th:abstracttheorem}. For instance, the nonlocal character of the problem requires the 
path $\gamma$ to satisfy a convexity inequality for arbitrary pairs of points in $\mathbb{R}^N$.  

\smallbreak
With respect to systems of equations we present a new proof of a known uniqueness result \cite{Dalmasso2000, montenegro} for positive solutions of Hamiltonian elliptic system
in the sublinear case
\begin{equation}\label{eq:hs}
- \Delta u = |v|^{q-1}v\,, \quad 
 -\Delta v = |u|^{p-1}u \,, \qquad 
 p, q> 0\,, \quad p\cdot q < 1 \,.
\end{equation}
These systems can be seen as a generalization of the biharmonic equation, since with $q=1$, $u$ solves the fourth-order elliptic equation
\begin{equation}
\Delta^2 u = |u|^{p-1}u\,, 
\qquad 
0 < p <1.
\end{equation}
Due to their structure, they pose many mathematical challenges, and they can be treated by using several variational frameworks, each one with its advantages and disadvantages (we refer to the survey \cite{BonheuredosSantosTavares} for more details). To treat Hamiltonian system by our methods, we  define an appropriate functional by using the dual method approach, which goes back to \cite{ClementvanderVorst}.
Then the principal part of the functional contains the inverse of the Laplacian rather than differentials, and therefore 
new convexity inequalities are needed. 
 Another obstacle is a proper choice of a path $\gamma$ and the verification of the local Lipschitz property in a multi-component setting. 

\smallbreak
We include an example involving the
quasilinear defocusing Schr\"odinger equation 

\begin{equation}\label{eq:QUASILINEAR_intro}
 -\Delta u-u\Delta u^2+ V(x) u + u^3 =\omega u 
 \end{equation}
 on both bounded domains and on $\mathbb{R}^N$, where $V$ is an appropriate potential and 
 $\omega$ is either fixed or a Lagrange multiplier when the mass (i.e. the $L^2$--norm) is fixed. 
This equation appears when looking for standing waves of a Schr\"odinger type equation, and is a particular case of a more general problem appearing in many physical phenomena such as plasma physics and fluid dynamics, or condensed matter theory, see \cite{ColinJeanjeanSquassina, PoppSchmittWang} for a detailed list of physical references.
 The main difficulty in the application of our main results when working in $\mathbb{R}^N$ is to obtain a comparison for positive weak solutions of \eqref{eq:QUASILINEAR_intro}, which is needed for the proof of 
the local Lipschitz continuity of $\gamma$. Such comparison can be derived from new sharp decay estimates, 
\[
u(x)\sim |x|^\frac{\omega-N}{2}e^{-\frac{1}{2}|x|^2} \qquad \text{ as } |x|\to \infty,
\]
that are proved for a transformed problem with a
simplified principal part.

\smallbreak
Finally, we show an application to the Gross-Pitaevskii system 
\begin{equation}\label{eq:GP_intro}
-\Delta u_i +V(x) u_i +u_i \sum_{j=1}^k \beta_{ij} u_j^2=\omega_i u_i \qquad i=1,\ldots, k
\end{equation}
which arises as a model for standing waves of Bose-Einstein Condensates \cite{Timmermans} or in Nonlinear Optics \cite{AA}. We treat the case when $\omega_1,\ldots, \omega_k$ are fixed, as well as the case when the mass of each $u_i$ is fixed, and the parameters appear as Lagrange multipliers. As in the previous example, when $\Omega=\mathbb{R}^N$ we need the following sharp decay estimates 
for positive solutions $(u_1,\ldots, u_k)$ of \eqref{eq:GP_intro},
\begin{equation}\label{eq:decay_NLSsystem}
u_i(x)\sim |x|^{\frac{\omega_i-N}{2}}e^{-\frac{1}{2}|x|^2} \qquad \text{ as } |x|\to \infty.
\end{equation}
However, additional difficulties stem from the fact that critical points of the associated functional might have some trivial components, and 
therefore are not comparable. Moreover, when the mass of the $u_i$ is fixed, the parameters $\omega_i$ may depend on the solution. In that case, the sharp decay estimates yield that positive solutions are not comparable and our method does not apply, but we include an alternative proof for completeness which also relies on \eqref{eq:decay_NLSsystem}.%

\smallbreak
The paper is organized as follows. In Section \ref{sec:proofmaintheorem} we prove Theorem \ref{th:abstracttheorem} and Corollary \ref{cor:abstractcor}.
We collect general auxiliary statements and inequalities needed throughout of the paper, in Section \ref{sec:aux-res}. Section \ref{sec:application2order}
contains applications to second order problems and Section \ref{section:meancurvatureoperators} to mean curvature operators. 
Problems involving the fractional Laplacian are discussed in Section \ref{section:fractionalLaplacian} and the problems regarding 
Hamiltonian systems are in Section \ref{sec-HS}.
Our study of 
Schr\" odinger
equations and Gross-Pitaevskii systems can be found in Section \ref{sec:RN}.

\section{Proof of Theorem \ref{th:abstracttheorem} and Corollary \ref{cor:abstractcor}} \label{sec:proofmaintheorem}

\begin{proof}[Proof of Theorem {\rm{\ref{th:abstracttheorem}}}] \emph{i) \,} For a contradiction, suppose that there exist $u, v \in A$ such that $I(v) < I(u)$ and set $N := I(v) - I(u) <0$. Then, with $\gamma$ as in the hypotheses of this theorem, 
\begin{equation}\label{eq:00}
 I(\gamma(t)) - I(\gamma(0)) \leq (1-t) I(u) + t \, I(v) - I(u) = t \, N \qquad
   \textrm{ for all } \, t  \in (0,1)\,,
\end{equation}
 that is,
 \begin{equation}\label{eq:1proof}
 \frac{I(\gamma(t)) - I(u)}{t} \leq N <0\qquad   \textrm{ for all } 
  t \in (0,1),
 \end{equation}
and in particular $\gamma(t) \neq u$ for all $t \in (0,1)$. On the other hand, since $\gamma$ is locally Lipschitz at $0$, there exist $\delta > 0$ and $C> 0$ such that
\begin{equation}\label{eq:2proof}
 \| \gamma(t) - u \| \leq  C t \qquad   \textrm{ for all }  t \in [0,\delta].
\end{equation}
Since $I$ is Fr\' echet differentiable and $u$ is a critical point, then 
\eqref{eq:1proof} and \eqref{eq:2proof} yield
\begin{align*}
 0 = \displaystyle {\lim_{t \to 0} \frac{|I(\gamma(t)) - I(u) - I'(u)(\gamma(t) -u) |}{\| \gamma(t) -u\|}} =  \displaystyle{\lim_{t \to 0} \frac{| I(\gamma(t)) - I(u) |}{\| \gamma(t) -u \|} \geq \frac{|N|}{C} >0},
\end{align*}
which is a contradiction. Therefore, $I$ is constant on $A$ and i) follows.

\smallbreak
\noindent \emph{ii)\, } Let $j:[0,1] \to \mathbb{R}$ be defined as $j(t) = I(\gamma(t))$. For every $t\in(0,1)$ such that $\gamma(t) \neq \gamma(0)$ we can write
\[
\frac{|j(t)- j(0)|}{t}  = \frac{|I(\gamma(t)) - I(u) - I'(u)(\gamma(t) - \gamma(0)) |}{\|\gamma(t) - \gamma(0)\| } \frac{\|\gamma(t)- \gamma (0)\|}{t}
\] 
and we infer that $j'(0)=0$.  Then $j:[0,1] \to \mathbb{R}$ is convex, $j(0) =j(1)$, $j'(0) =0$, and therefore constant on $[0,1]$. Indeed,  since $j$ is convex
\[
\frac{j(h) - j(0)}{h} \leq \frac{j(t+h) - j(h)}{t}, \qquad 
\textrm{ for all } t \in (0,1), h\in (0, 1-t).
\]
Then, taking the limit as $h \to 0^+$, and using $j'(0) = 0$, we infer that $j(0)\leq j(t) \leq (1-t)j(0)+ t j(1) = j(0)$.
Part \emph{iii)\, } immediately follows from \emph{ii)\,}. 
\end{proof}

\begin{proof}[Proof of Corollary {\rm{\ref{cor:abstractcor}}}]
Recall that $u$ is a critical point of $I |_S$ if and only if there exists $\lambda$ (depending on $u$) in the dual space of $Y$ such that $u$ is a critical point of the functional $J: X \rightarrow {\mathbb R}$ defined by
\[
J = I -  \lambda\circ R.
\]
Then the proof follows by applying  the arguments in the proof of Theorem \ref{th:abstracttheorem}
to the functional $J$ and
taking into account that $J\circ \gamma = I\circ \gamma$ since $\gamma(t) \in S$ for all $t\in [0,1]$.
Observe that in the proof of Theorem \ref{th:abstracttheorem}
the fact that $v$ is a critical point is not actually needed.
\end{proof}

In case $A$ contains a local minimum of $I$, we present a similar \textemdash but non-equivalent\textemdash \vspace{5pt} version of Theorem \ref{th:abstracttheorem}. Observe that in this version $t \mapsto I(\gamma(t))$ is not assumed to be convex and that one could also state the corresponding version of Corollary \ref{cor:abstractcor} within this weaker setting. The proof follows the same arguments as in the proof of Theorem \ref{th:abstracttheorem} and it is omitted. 

\vspace{10pt}

\begin{theorem}\label{1.1'} Let $X$ be a normed space, $I: X \rightarrow {\mathbb R}$ be a Fr\'echet differentiable functional and $A \subset X$ be a nonempty subset of critical points of $I$. Suppose that given $u,v \in A$, with $u\neq v$, there exists $\gamma \in C([0,1],X)$ such that
 \begin{enumerate}[a)]
 \item $\gamma(0) = u$, $\gamma(1)=v$.
 \item $\gamma$ is locally Lipschitz at $t=0$, that is, $\|\gamma(t) - u\| \leq C t$ for each $t \in [0, \delta]$ and some $\delta > 0$.
 \item $I(\gamma(t)) \leq (1-t)I(u) + t \, I(v)$ for all $t \in (0,1)$. 
 \end{enumerate}
 Then,
 \begin{enumerate}[i)]
 \item  $I$ is constant on $A$. 
 \item if $A$ contains a local minimum $u_0$ of $I$ 
 and the strict inequality holds at condition c), 
 then $A = \{ u_0\}$. 
 \end{enumerate}
\end{theorem}

\vspace{10pt}
We point out that in Theorem \ref{1.1'} the existence of a local minimum is sharp in order to prove that $A$ is a singleton. Indeed, let $B \subset {\mathbb R}^N$, with $N \geq 1$, be the unit ball in ${\mathbb R}^N$ centered at the origin and consider the H\'enon equation \cite{Henon}
\begin{equation}\label{eq:henon}
-\Delta u  = |x|^{\alpha}|u|^{q-2}u \ \ \text{in} \ \ B, \quad u = 0 \ \ \text{on} \ \ \partial B,
 \end{equation}
with $\alpha>0$, $2 < q$ if $N \in \{1, 2\}$ and $2< q< \frac{2N}{N-2}$ if $N \geq3$. Then the classical solutions of \eqref{eq:henon} are critical points of 
\[
I_{\alpha}(u) = \frac{1}{2} \int_B |\nabla u|^2 \, dx  - \frac{1}{q}\int_{B} |x|^{\alpha} |u|^{q} \, dx, \quad u \in H^1_0(B).
\]
Let $\alpha > 0$ be large enough such that least energy solutions \textemdash L.E.S. for short\textemdash \vspace{5pt} of \eqref{eq:henon} are not radially symmetric; see \cite{SmetsSuwillem, ByeonWang}. Set $A$ as either
 \[
 \{ u; \, u \ \ \text{is a L.E.S. of \eqref{eq:henon}}\} \quad \text{or} \quad \{ u; \, u \ \ \text{is a L.E.S. of \eqref{eq:henon} and $u>0$ in $\Omega$}\}.
\]
Given $u, v\in A$, with $u \neq v$, consider the path
\[
\gamma (t) = \left\{ 
\begin{array}{l}
(1-2t)u,\ \  t\in[0,1/2],\vspace{5pt} \\  (2t-1) v, \ \ t \in [1/2,1].
\end{array}
\right.
\]
Taking into account the characterization of the Nehari manifold
\[
\mathcal{N}_{\alpha} = \{ u \in H^1_0(B) \backslash \{0\}; \, I_{\alpha}'(u)u = 0\},
\]
we infer that all hypotheses a)-c) of Theorem \ref{1.1'} are satisfied with the strict inequality at condition c). However, for 
$N \geq 2$, and non-radial $u^* \in A$, the set $A$ contains infinitely many critical points 
$\{ u^* \circ O; \, O \in SO(N) \} \subset A$
of $I_{\alpha}$ which are all of mountain pass type. Observe also that $t \mapsto I(\gamma(t))$ is not convex, and therefore Theorem 
\ref{th:abstracttheorem} is not violated. Also, for $v_1,v_2\in \{ u^* \circ O; \, O \in SO(N) \}$ there is clearly a path between $v_1$ and $v_2$ along which $I$ is constant, so mere convexity is in general not enough to guarantee the uniqueness in Theorem \ref{th:abstracttheorem}.

\section{Preliminary results} \label{sec:aux-res}

In this section we collect general results used throughout the paper that help to verify the assumptions of Theorem \ref{th:abstracttheorem} and Corollary 
\ref{cor:abstractcor}. Let $\Omega \subseteq \mathbb{R}^N$ be a domain (bounded or unbounded) and fix $u,v:\Omega\to \mathbb{R}$ belonging to an appropriate space $W$  specified below. We consider paths $\gamma : [0, 1] \to W$ of the form
\begin{equation}
\gamma(t)(x) := Q^{-1} ((1 - t)Q(u(x)) + tQ(v(x))), \qquad  t \in [0, 1], x \in \Omega \,,
\end{equation}
 where $Q :[0, \infty) \to \mathbb{R}$ is an increasing, and therefore invertible function.  For simplicity and without loss of generality we also assume that $Q(0)=0$.  Then we have 
\begin{equation}\label{eq:ggm}
Q(\gamma(t)) := (1 - t)Q(u) + tQ(v) \,,
\end{equation}
where we suppressed the dependence on $x$ for simplicity. The model function that satisfies all assumptions below is $Q(z) = z^p$ for $p > 1$.  

First, we provide a general  criterion for the Lipschitz continuity of $\gamma$ at $t=0$. Note that even in the
model case $Q(z) = z^{p}$ for $p>1$ we have to assume that $u$ and $v$ are comparable, that is, for some $\delta \geq 1$ we have on $\Omega$
\begin{equation}\label{eq:hipderivative2}
u, v > 0\,, \qquad 
\frac{1}{\delta} \leq \frac{u}{v} \leq \delta \,.
\end{equation} 
Otherwise if say $u \equiv 0$, then $\gamma(t) = t^{\frac{1}{p}}v$, which is not a locally Lipschitz function at $t = 0$. This comparability is
assumed to be preserved by $Q$ as specified in the following lemma.

\begin{lemma}\label{lemma:gpg}
Let $W$ stand for either $L^p(\Omega)$, $W^{1,p}_0(\Omega)$, or $W^{1,p}(\Omega)$ with $p \geq 1$. Fix $Q \in C([0, \infty)) \cap C^1((0, \infty))$
such that 
\begin{itemize}
\item[a)] $Q' > 0$ on $(0, \infty)$,
\item[b)] for each $c_0 > 0$ there is $c_2 > 0$ such that 
\begin{equation}
\frac{1}{c_2} \leq \frac{Q'(z_1)}{Q'(z_2)} \leq c_2, \qquad 
\textrm{whenever } \frac{1}{c_0} \leq \frac{z_1}{z_2} \leq c_0 \,, \quad z_1, z_2 > 0 \,.
\end{equation}
\end{itemize}
If  $u, v \in W$ satisfy \eqref{eq:hipderivative2} for some $\delta \geq 1$, 
then $\gamma:[0,1] \to W$ defined by \eqref{eq:ggm} is locally Lipschitz at $t = 0$, provided $W = L^p(\Omega)$. If in addition 
\begin{itemize}
\item[c)] $Q \in C^2((0, \infty))$ and there is $\delta_1 > 1$ and $c_3>0$ such that 
\begin{equation}
\qquad \frac{1}{c_3} \leq \frac{Q''(z_1)}{Q''(z_2)} \qquad \textrm{whenever} \quad 1 \leq \frac{z_1}{z_2} \leq \delta_1 \,,
\end{equation}
then $\gamma:[0,1]\to W$ is locally Lipschitz at $t = 0$, for any choice of $W$.
\end{itemize}

\begin{remark}
From the proof of the Lemma \ref{lemma:gpg} immediately follows
that the statement holds true for weighted Lebesgue and Sobolev 
spaces. 
\end{remark}
\end{lemma}

\begin{proof} 
To simplify the notation we drop the dependence of functions on $x$.
Clearly $\gamma(0) = u$, $\gamma(1)=v$ and, since $Q$ is increasing so is $Q^{-1}$ and we have 
\begin{equation}
\min\{u,v\} \leq \gamma(t) \leq \max\{u,v\} \,, \quad \textrm{ for all } \, t \in [0,1].
\end{equation}
Let us prove that $\|\gamma(t) - u \|_W \leq C t$ for $t \in [0, 1]$.
By the defintion of $\gamma$ we have (whenever $u(x) \neq v(x)$ and $t \neq 0$)
\begin{equation}\label{eq:beq}
Q(v) - Q(u) =  \frac{Q(\gamma(t)) - Q(u)}{t} = 
\frac{Q(\gamma(t)) - Q(u)}{\gamma(t) - u}
 \frac{\gamma(t) - u}{t} \,,
\end{equation}
and by the Mean-value Theorem
\begin{equation} \label{eq:dfxe}
Q(v) - Q(u) = Q'(\xi) (v - u), \qquad 
\frac{Q(\gamma(t)) - Q(u)}{\gamma(t) - u} = Q'(\eta) \,,
\end{equation}
where $\xi$ is pointwise between $u$ and $v$, and $\eta$ between $u$ and $\gamma(t)$. In particular,
\[
\frac{1}{\delta} \leq \frac{\min\{u,v\}}{\max\{u,v\}} \leq \frac{\xi}{\eta} \leq \frac{\max\{u,v\}}{\min\{u,v\}} \leq \delta.
\]
Thus from
\begin{equation}\label{eq:fpr}
\frac{\gamma(t) - u}{t} = \frac{Q'(\xi)}{Q'(\eta)} (v - u) \,,
\end{equation}
and b), we have $\|\gamma(t) - u\|_{L^p(\Omega)} \leq C t$, whenever $u, v \in L^p (\Omega)$. 

To treat the Sobolev spaces, first observe from b), c), $Q' > 0$, and the Mean-value Theorem 
that one has, for $z_1>0$,
\begin{equation}
c_2 - 1 \geq \frac{|Q'(\delta_1 z_1) - Q'(z_1)|}{Q'(z_1)} = (\delta_1 - 1) \frac{|Q''(w)|}{Q'(z_1)} z_1 = 
(\delta_1 - 1) \frac{Q''(w)}{Q''(z_1)} \frac{|Q''(z_1)|}{Q'(z_1)} z_1 
\geq  \frac{\delta_1 - 1}{c_3} \frac{|Q''(z_1)|}{Q'(z_1)} z_1 \,,
\end{equation}
where the last inequality holds since $z_1 < w < \delta_1 z_1$. Consequently for each $z_1 > 0$
\begin{equation}\label{eq:scd}
 \frac{|Q''(z_1)|}{Q'(z_1)} z_1 \leq C \,.
\end{equation}
After differentiating \eqref{eq:ggm} we have that
\begin{equation}
Q'(\gamma(t)) (\nabla \gamma(t) - \nabla u)
= (Q'(u) - Q'(\gamma(t))) \nabla u + t(Q'(v) \nabla v - Q'(u)\nabla u) =: T_1 + T_2\,.
\end{equation}
Since $\gamma(t)$ and $u$, $v$ are comparable, by the Mean-value Theorem, b), \eqref{eq:fpr}, and \eqref{eq:scd} we obtain 
\begin{equation}
\frac{|T_1|}{t Q'(\gamma(t))} \leq \frac{|Q''(\zeta)|}{Q'(\gamma(t))} \frac{|u - \gamma(t)|}{t}|\nabla u|
\leq C \frac{Q'(\zeta)}{Q'(\gamma(t))} \frac{Q'(\xi)}{Q'(\eta)} \frac{|u - v|}{\zeta} |\nabla u| \leq C  |\nabla u| \,,  
\end{equation}
where in the last inequality we used that $\zeta$ lies between $u$ and $\gamma(t)$, and therefore $\gamma(t)$, $\zeta$, $\xi$, and $\eta$ are all mutually comparable in the sense of \eqref{eq:hipderivative2}.
Finally, 
\begin{equation}
\frac{|T_2|}{t Q'(\gamma(t))} \leq C (|\nabla u| + |\nabla v|)
\end{equation}
follows immediately from b). In conclusion, we have
\[
\frac{|\nabla \gamma(t)-\nabla u|}{t}\leq \kappa (|\nabla u|+|\nabla v|)
\] 
and the local Lipschitz continuity of $\gamma$ at $t=0$ follows when $W$ is either $W^{1,p}(\Omega)$ or $W^{1,p}_0(\Omega)$.
\end{proof}

An immediate application of the previous lemma is the following.

\begin{corollary}\label{corollary_of_lemma:gpg}
Let $W$ stand for either $L^p(\Omega)$, $W^{1,p}_0(\Omega)$, or $W^{1,p}(\Omega)$ with $p \geq 1$ and take $r >1$. If  $u$ and $v$ satisfy \eqref{eq:hipderivative2} for some $\delta \geq 1$, then the path $\gamma:[0,1]\to W$ defined by
\[
\gamma(t)=((1-t)u^r+tv^r)^{\frac{1}{r}}
\]
is locally Lipschitz at $t=0$.
\end{corollary}
\begin{proof}
Just apply the previous lemma to $Q(z)=z^r$, with $r > 1$.
\end{proof}

Note that if $u$ and $v$ satisfy $0 < c \leq u, v < C$ in $\overline{\Omega}$, then \eqref{eq:hipderivative2} is clearly satisfied for $\delta = C/c$. 
If $u$ and $v$ attain zero Dirichlet boundary conditions, we have the following well known lemma, whose assumptions are usually checked with 
the help of Hopf's Lemma.

\begin{lemma}\label{lemma:gradientboundarycomparison}
Let $\Omega \subset {\mathbb R}^N$, $N\geq 1$, be a bounded smooth domain. Suppose that $u, v\in C^1(\overline{\Omega})$ satisfy
\begin{enumerate}[a)]
 \item $u, v > 0$ in $\Omega$ and $u=v =0$ on $\partial \Omega$; \vspace{5pt}
 \item $\displaystyle{\frac{\partial u}{\partial \nu} < 0}$ and $\displaystyle{\frac{\partial v}{\partial \nu} < 0}$ on $\partial \Omega$.
\end{enumerate}
Then there exists $\delta \geq1$ such that $\delta^{-1} v <u < \delta v$ in $\Omega$.
\end{lemma}
\begin{proof}
The proof is standard and hence omitted.
\end{proof}

Next, we turn our attention to the assumption c) of Theorem \ref{th:abstracttheorem} for paths of type \eqref{eq:ggm}.
In our examples, the principal part of the functional $I$ usually has the form $\int_{\Omega} M(|\nabla u|)$ and the following result
proves its convexity.

\begin{lemma}\label{lem:gin}
Let $u, v \in W^{1, \infty}(\Omega) \cap W$ with $u, v > 0$ in $\Omega$ and let {$Q,M\in C([0, \infty)) \cap C^1((0, \infty))$} such that $Q(0)=0$ and $Q', M' > 0$ on $(0,\infty)$. Let $\gamma(t)=Q^{-1}((1-t)Q(u)+tQ(v))$ and denote
\begin{equation}
F_1 := Q'\circ Q^{-1}, \qquad F_2 := M^{-1} \,.
\end{equation}
If for some $\Gamma \in (M(0),\infty]$ the function $F : (z_1, z_2)  \mapsto F_1(z_1)F_2(z_2)$ is concave on $(0, \infty) \times (M(0), \Gamma)$, then 
for each $|\nabla u|, |\nabla v| \in M^{-1}([M(0), \Gamma])$
we have a pointwise inequality
\begin{equation} \label{eq:beqq}
M(|\nabla \gamma(t)|) \leq (1 - t)M(|\nabla u|) + t M(|\nabla v|), \quad \textrm{ for all } \, t\in[0,1]
\end{equation}
and $t \mapsto M(|\nabla \gamma (t,\cdot)|)$ is convex. If $F$ is strictly concave on $(0, \infty) \times (M(0), \Gamma)$, then 
for each $t \in (0, 1)$,
$|\nabla u| \in M^{-1}([M(0), \Gamma])$, and $|\nabla v| \in M^{-1}((M(0), \Gamma))$, the strict inequality holds in \eqref{eq:beqq}, and  
$t \mapsto M(|\nabla \gamma (t,\cdot)|)$ is strictly convex.
\end{lemma}

\begin{proof} 
By differentiating \eqref{eq:ggm} we have 
\begin{equation}
Q'(\gamma(t)) \nabla \gamma(t) = (1- t)Q'(u) \nabla u + tQ'(v) \nabla v \,,
\end{equation}
and consequently, by using triangle inequality and the invertibility of $Q$ and $M$ 
\begin{align}\label{eq:tiq}
Q'(\gamma(t)) |\nabla \gamma(t)| &\leq (1 - t) Q'(u) |\nabla u| + t Q'(v) |\nabla v| \\
&= (1 - t) Q'\circ Q^{-1}\circ Q(u) M^{-1}\circ M(|\nabla u|) + 
t Q'\circ Q^{-1}\circ Q(v) M^{-1}\circ M(|\nabla v|) \\
&= (1 - t)F_1 (Q(u)) F_2(M(|\nabla u|)) + t F_1 (Q(v)) F_2(M(|\nabla v|)) \,.
\end{align}
Since $F$ is concave on $(0, \infty) \times (M(0), \Gamma)$, it is concave on $(0, \infty) \times [M(0), \Gamma]$ and 
for any $|\nabla u|, |\nabla v| \in M^{-1}([M(0), \Gamma])$ we have 
\begin{align}\label{eq:cnvx}
Q'(\gamma(t)) |\nabla \gamma(t)| &\leq F_1((1 - t) Q(u) + t Q(v)) F_2((1 - t)M(|\nabla u|) + t M(|\nabla v|))\\
&= F_1 (Q(\gamma(t))) F_2 ((1 - t)M(|\nabla u|) + t M(|\nabla v|)) \\ 
&= Q'(\gamma(t))  M^{-1}((1 - t)M(|\nabla u|) + t M(|\nabla v|)) 
\end{align}
and \eqref{eq:beqq} follows. Note that \eqref{eq:beqq} also implies that $|\nabla \gamma(t)| \in M^{-1}([M(0),\Gamma])$ for each $t \in[0,1]$.

To prove the convexity of $t \mapsto M(|\nabla \gamma (t)|)$,
fix $t_1, t_2, \theta \in [0,1]$ and set $\gamma =: \gamma_{uv}$ to emphasize its dependence on the endpoints $u$ and $v$. 
Then it is easy to verify that 
\begin{equation}\label{eq:globalargument}
\gamma_{uv}((1-\theta)t_1 + \theta t_2) = \gamma_{U_1U_2}(\theta),\text{ where } U_i \text{ is defined by } Q(U_i) = (1-t_i)Q(u) + t_i Q(v),\ i = 1, 2.
\end{equation} Then, by \eqref{eq:beqq} applied to $\gamma_{U_1U_2}$
\begin{align}
M(|\nabla \gamma_{uv}((1-\theta)t_1 + \theta t_2) |) &= M(|\nabla \gamma_{U_1U_2}(\theta)|) 
\leq (1 - \theta) M(|\nabla \gamma_{U_1U_2}(0)|) + \theta M(|\nabla \gamma_{U_1U_2}(1)|) 
\\
&= (1 - \theta) M(|\nabla \gamma_{uv}(t_1)|) + \theta M(|\nabla \gamma_{uv}(t_2)|) 
\,,
\end{align}
and the convexity follows. 

To prove the strict convexity, first observe that for each $t \in(0,1)$
\begin{equation}\label{eq:acvx}
(1 - t) F(z) + t F(\bar{z}) <  F((1 - t) z + t \bar{z}) \,,
\end{equation}
where $z \in (0, \infty) \times [M(0), \Gamma]$ and $\bar{z} \in (0, \infty) \times (M(0), \Gamma)$. Indeed, if not, then  
there exist $t_0 \in (0, 1)$ such that equality holds in \eqref{eq:acvx}
with $t$ replaced by $t_0$. Then, by the concavity of $F$, we obtain that the equality holds in \eqref{eq:acvx} for each 
$t \in [0, 1]$, and therefore $F$ is linear along the segment connecting $z$ and $\bar{z}$. Since such segment (except one of the endpoints) lies in 
$(0, \infty) \times (M(0), \Gamma)$, we obtain a contradiction to the strict concavity of $F$. 

Finally, the strict inequality in \eqref{eq:beqq} is a consequence of the strict concavity in \eqref{eq:cnvx}, and strict convexity of 
$t \mapsto M(|\nabla \gamma(t, \cdot)|)$ follows as above. 
\end{proof}

\begin{remark}
To our best knowledge, a special case of the following lemma with $Q(z) = M(z) = z^2$
first appeared in  \cite{BenguriaBrezisLieb}. 
The case $Q(z) = M(z) = z^p$ was treated in \cite[Lemma 1]{DiazSaa}, 
see also \cite{BelloniKawohl} and \cite[Chapter 2]{Reichel}. Note that our general results 
require a completely different proof based on concavity, which is in a sense optimal; see Remark \ref{rmk:opt} below. At Section {\rm{\ref{subsec:quasilinear}}} we will consider the case $M(z) = z^2$ and $Q(z) = f^2$ where $f$ is the odd function such that
\[
f'(t)=\frac{1}{\sqrt{1+2f^2(t)}} \quad \text{in } (0,\infty), \quad f(0)=0.
\]
We work with classical derivatives to avoid methods of Orlicz spaces. However, the arguments hold true whenever 
 the expressions are defined, cf. proof of Lemma {\rm{\ref{lemma:strictconvexity}}} below for the discussion on weak derivatives. 
\end{remark}

\begin{remark}\label{rmk:opt}
Our assumptions are in some sense optimal, since if $F$ is not concave, we obtain an opposite inequality in \eqref{eq:cnvx} at some points. 
Besides the triangle inequality, this is the only estimate used in the proof, and therefore \eqref{eq:beqq} is not expected
to hold true in general. 
\end{remark}

\begin{remark}\label{remark:concavity}
To verify the concavity of the function $(z_1, z_2) \mapsto F_1(z_1)F_2(z_2)$, since $F_1,F_2\geq 0$,
 one needs both $F_1$ and $F_2$ to be concave and 
the determinant of the Hessian matrix to be non-negative, that is, 
\begin{equation}
F_1F_1''F_2F_2'' \geq (F_1'F_2')^2 \,,
\end{equation}
where $F_i$ depends on $z_i$. If  $F_i'$ does not vanish we require
\begin{equation} \label{eq:opp}
1 \leq \frac{F_1F_1''}{(F_1')^2}\frac{F_2F_2''}{(F_2')^2} = \left(\left(\frac{F_1}{F_1'}\right)' - 1 \right) \left(\left(\frac{F_2}{F_2'}\right)' - 1 \right) \,,
\end{equation}
where the first parenthesis depends only on $z_1$ whereas the second one depends only on $z_2$. Recall that $F_2 := M^{-1}$ is given by the problem, being associated to the principal part of the functional. 

As such \eqref{eq:opp} provides partially optimal sufficient conditions on $Q$. Indeed, if for given $M$  
the second parenthesis changes sign, there is no no-trivial $Q$ yielding the desired convexity. 
However, if for example the second parenthesis is bounded from below by $\frac{1}{c_0} > 0$, then 
we can explicitly solve the differential inequality to obtain
\begin{equation}
F_1 \leq c_2 ( (c_0 + 1)z_1 + c_1)^\frac{1}{c_0 + 1}
\end{equation}
for some constants $c_1, c_2 \geq 0$. Recall that $F_1 = Q' \circ Q^{-1}$ and we obtain a differential inequality for $Q$
\begin{equation}
Q' \leq c_2 ( (c_0 + 1)Q + c_1)^\frac{1}{c_0 + 1}, \qquad Q(0) = 0 \,.
\end{equation}
If $c_1 = 0$, or equivalently $Q'(0) = 0$, this inequality can be solved explicitly. 
\end{remark}

If the function $(x, y) \mapsto F_1(x)F_2(y)$ from Lemma 
\ref{lem:gin} is concave, but not strictly concave, the equality in \eqref{eq:beqq} is a much more subtle issue and, in general, it is achieved at $(u, v)$ with
a nontrivial relation.  However, we show an important example that we can treat explicitly. Many of the ideas in the proof can be used also in a more general setting.

\begin{lemma}\label{lemma:strictconvexity}
Take $u, v \in W$, where $W$ stands for either $W^{1,p}_0(\Omega)$ or $W^{1,p}(\Omega)$ with $p > 1$, with $u,v>0$ in $\Omega$. If $Q (z) = M(z) = z^p$ and $\gamma$ is as in \eqref{eq:ggm}, that is, $\gamma(t)=\left((1-t) u^p+t v^p\right)^{1/p}$, then the weak derivatives satisfy
\begin{equation} \label{eq:pinq}
|\nabla \gamma(t)|^p \leq (1 - t)|\nabla u|^p + 
t |\nabla v|^p
\end{equation}
and $t \mapsto |\nabla \gamma (t, x)|^p$ is convex. Moreover, if $u, v \in C(\overline{\Omega}) \cap W$ with $u, v$ being linearly independent functions, then
$t \mapsto \|\nabla \gamma (t)\|_{L^p(\Omega)}^p$ is strictly convex. 
\end{lemma}

\begin{proof}
It is easy to show that under our assumptions $u^p, v^p \in W^{1, 1}(\Omega)$ and the proof of Lemma 
\ref{lem:gin} can be repeated line by line with pointwise derivatives replaced by weak ones.

Let $q$ be the conjugate exponent of $p$, that is 
$q$ satisfies $1/p + 1/q = 1$. If $F_1$ and $F_2$ are
as in Lemma \ref{lem:gin}, then $F_1(z_1) =p z_1^{1/q}$
and $F_2(z_2) = z_2^{1/p}$. 

Clearly $F_1'' < 0$, $F_2''<0$ and it is easy to verify that  the Hessian is equal to 
\begin{equation}
H(z_1, z_2) = 
\frac{z_1^{\frac{1}{q} - 2} z_2^{\frac{1}{p} - 2}}{pq}\frac{1}{q}
\left(
\begin{array}{cc}
-z_2^2 & z_1z_2 \\
z_1z_2 & -z_1^2
\end{array}
\right)
\end{equation}
with eigenvalues $0$ and $-(z_1^2 + z_2^2)$ corresponding to the eigenvectors $(z_1, z_2)^T$ and $(-z_2, z_1)^T$ respectively. 
In particular, $(z_1, z_2) \mapsto F_1(z_1)F_2(z_2)$ is concave and \eqref{eq:pinq} and the convexity of $t \mapsto |\nabla \gamma (t, x)|^p$ follows from Lemma \ref{lem:gin}. However, it is not strictly concave and we need a careful inspection of the proof of Lemma \ref{lem:gin} to prove that $t \mapsto \|\nabla \gamma (t)\|_{L^p(\Omega)}^p$ is strictly convex.

Since $Q$ and $Q^{-1}$ are monotone and homogeneous ($Q(\lambda t) = \lambda^p Q(t)$), then the linear independence of 
$u$ and $v$ implies that 
$Q(u)$ and $Q(v)$ are linearly independent, and consequently $U_1=\gamma(t_1)$ and $U_2=\gamma(t_2)$  are linearly independent for 
all $t_1, t_2 \in [0,1]$ with $t_1\neq t_2$. Hence, arguing as in \eqref{eq:globalargument} of the proof of Lemma \ref{lem:gin}, it is enough to prove that
\begin{equation}\label{ineq:incv18}
\|\nabla \gamma (t)\|_{L^p(\Omega)}^p < (1-t)\|\nabla u\|_{L^p(\Omega)}^p + t \|\nabla v\|_{L^p(\Omega)}^p \quad \textrm{ for all }  t \in(0,1).
\end{equation}
Suppose that, for some $t\in(0,1)$, \eqref{ineq:incv18} does not hold. Hence $ \|\nabla \gamma (t)\|_{L^p(\Omega)}^p = (1-t)\|\nabla u\|_{L^p(\Omega)}^p + t \|\nabla v\|_{L^p(\Omega)}^p$. First, the equality holds in the triangle inequality \eqref{eq:tiq} if and only if there is $\alpha : \Omega \to [0, \infty)$
such that 
\begin{equation}\label{eq:gradientsLD2}
\nabla v  = \alpha \nabla u \,. 
\end{equation} 
Second, the equality holds in the concavity inequality \eqref{eq:cnvx} if and only if the vector connecting the points $(Q(u), M(|\nabla u|))$
and $(Q(v), M(|\nabla v|))$ is parallel to the eigenvector corresponding to the zero eigenvalue at the point $(Q(u), M(|\nabla u|))$. Equivalently, 
there is $\beta : \Omega \to [0, \infty)$ such that 
\begin{equation}
\beta^p Q(u) = Q(v)\,, \qquad \beta^p M(|\nabla u|) = M(|\nabla v|) \,.
\end{equation} 
With our choice of $M$ and $Q$ we have 
\begin{equation}
\beta u = v \,, \qquad \beta |\nabla u| = |\nabla v| \,,
\end{equation}
and a comparison with \eqref{eq:gradientsLD2} yields $\alpha = \beta$.
Therefore, since $u, v > 0$ and $u, v \in C(\overline{\Omega})$, and consequently uniformly positive on compact subsets of $\Omega$,
then $\alpha = \frac{v}{u}$ is locally a $C (\Omega) \cap W$ function. Furthermore, 
the weak derivatives of $\alpha$ satisfy
\[
\nabla \alpha = \nabla \left(\dfrac{v}{u}\right) = \dfrac{u \nabla v - v \nabla u }{ u^2} = \dfrac{\left(\alpha u - v  \right) \nabla u }{u^2} 
\equiv 0,
\]
which shows by Du Bois-Reymond Lemma that $\alpha$ is constant and concludes the proof.
\end{proof}

\section{Generalized $p$-Laplacian equations}\label{sec:application2order}

For a bounded smooth domain $\Omega \subset {\mathbb R}^N$ and $p>1$, consider the equation
\begin{equation}\label{eq:supergeneral}
- \, {\rm{div}} (h(|\nabla u|^p) |\nabla u|^{p-2} \nabla u)  = g(x,u) \ \ \text{in} \ \ \Omega, \qquad u = 0 \ \ \text{on} \ \ \partial \Omega \,.
\end{equation}
First, under general assumptions on $h$ and $g$, we 
prove a general theorem (see Theorem \ref{th:applicationp-laplace} below) and then 
in Section \ref{sec:app-p} we present a unified proof to many classical uniqueness theorems involving quasilinear elliptic problems. 
We remark that our results holds true for Neumann boundary conditions
with straightforward modifications in the proofs. 
In the particular case $h \equiv 1$, the uniqueness was already established in \cite{DiazSaa}. Our assumptions in particular include Allen-Cahn-type $p$-Laplacian problems, see Example \ref{example-Allen-Cahn}, and so extend some previous results of \cite{Berestycki1981, Berger1979} for the case of $p=2$, $h\equiv1$, $g \equiv k u - u^{q-1}$ with $q>2$. 

\smallbreak
Set
\begin{equation} \label{eq:gdf}
H(t) := \int_0^t h(s) ds \quad \text{and} \quad G(x,t) := \int_0^t g(x,s) ds,
\end{equation}
and assume:
\begin{enumerate}[(H1)]
 \item $h:[0, \infty) \rightarrow [0, \infty)$ is continuous, bounded, and non-decreasing.
 
 \vspace{5pt}
 
 \item The map $u \mapsto \int_{\Omega} G(x,u) \, dx$ is Fr\'echet differentiable in $W^{1,p}_0(\Omega)$ and its derivative evaluated at $v\in W^{1,p}_0(\Omega)$ is $\int_{\Omega} g(x,u) v\, dx$.
  \vspace{5pt}
\item For every $x \in \Omega$, the function $t \mapsto G(x, t^{1/p})$ is concave on $[0, \infty)$.\vspace{5pt}
\end{enumerate}
 \smallbreak
We say that $u \in W^{1,p}_0(\Omega)$ is a weak solution of \eqref{eq:supergeneral} if
\[
\int_{\Omega} h(|\nabla u|^p) |\nabla u|^{p-2}\nabla u \nabla v \, dx - \int_{\Omega} g(x,u) v \, dx = 0, \quad \textrm{ for all }  v \in W^{1,p}_0(\Omega),
\]
or equivalently, $u$ is a critical point of the Fr\'echet differentiable functional
\begin{equation}\label{eq:functionalsupergeneral}
I(u) = \frac{1}{p} \int_{\Omega} H(|\nabla u|^p )\,  dx - \int_{\Omega} G(x,u)\, dx, \quad u \in W^{1,p}_0(\Omega).
\end{equation}

\vspace{5pt}

Observe that by (H1) the function $H$ is convex, but since we do not assume that $G$ is concave, $I$ is not necessarily convex. We also suppose the following condition, which is related to \eqref{eq:hipderivative2}.
\begin{enumerate}
\item[(H4)] (Regularity and Global comparison) 
Every critical point $u \geq 0$ of  \eqref{eq:functionalsupergeneral} is $C\left(\overline{\Omega}\right)$. 
Given 
any positive critical points $u, v$ of \eqref{eq:functionalsupergeneral}, there exists $\delta \geq 1$ such that $\delta^{-1} v \leq u \leq \delta v$ in $\Omega$.
\end{enumerate}

\begin{theorem}\label{th:applicationp-laplace}
Assume {\rm{(H1)}}--{\,\rm{(H4)}}, let $I$ be as in \eqref{eq:functionalsupergeneral} and $A$ be the set of positive critical points of  \eqref{eq:functionalsupergeneral}.  If $A \neq \emptyset$, then the following holds:
 \begin{enumerate}[i)]
\item  $I$ is constant on $A$. 
\item If $h$ is increasing, or the function in {\rm{(H3)}} is strictly concave, then $A$ is a singleton. 
\item If $h>0$ on $(0, \infty)$, then $A \subset \{\alpha u_0;\, \alpha \in (0, \infty)\}$ for some $u_0\in A$. 
\item If we assume {\rm(H4)} only for $u \in A' \subset A$, then i)--iii) holds with $A$ replaced by $A'$.
\end{enumerate}
\end{theorem}

\begin{remark}\label{remark-truncation}
\begin{enumerate}[i)]
\item Note that $h > 0$ implies that $H$ is strictly increasing. Also, if $h$ is strictly increasing, then $H$ is strictly convex. 
\item Hypothesis {\rm{(H2)}} is satisfied for example if $g:\Omega \times {\mathbb R} \to {\mathbb R}$ is continuous and there exist $C> 0$ and $r>0$ with $r(N-p) \leq (p-1)N+p$, such that  \vspace{5pt}
\[
| g(x,t) |  \leq C (1 + |t|^r), \quad \textrm{ for all }  t \in {\mathbb R}, \, x \in \Omega.
\]

\item Theorem {\rm{\ref{th:applicationp-laplace}}} can be trivially extended to differentiable functionals
\begin{equation}
 u \in W  \mapsto \tilde{I}(u) =\tilde{H}(|\nabla u|^p) - \tilde{G}(u) \,,
\end{equation}
where $W$ is $W^{1, p}_0(\Omega)$ or $W^{1, p}(\Omega)$ and the critical points of $\tilde{I}$ satisfy {\rm{(H4)}}. 
Moreover, $\tilde{H} : L^1(\Omega) \to \mathbb{R}$ is non-decreasing (with respect to the cone of  positive function in $L^1$) and convex and 
$\tilde{G} : W \to \mathbb{R}$ satisfy {\rm{(H3)}} with $G$ replaced by $\tilde{G}$.  

\item In the proof of uniqueness of positive solutions for \eqref{eq:supergeneral}, it is sometimes assumed
(see e.g. \cite[(H2)]{DiazSaa}) that
\[
\frac{g(x,t)}{t^{p-1}} \ \ \text{is strictly decreasing in} \ \ (0, \infty),
\]
which implies {\rm{(H3)}}.

\item On bounded domains the global comparison of positive solutions (GC for short) introduced in {\rm{(H4)}} is a consequence of Hopf Lemma. Indeed, if
 $\frac{\partial u}{\partial \nu} < 0$ and $\frac{\partial v}{\partial \nu} < 0$ on $\partial \Omega$, then the GC follows from Lemma {\rm{\ref{lemma:gradientboundarycomparison}}}. 
 However, Hopf Lemma is not suitable to obtain GC if the problem is posed on ${\mathbb R}^N$. In this case the argument needs to be replaced by 
 sharp decay estimates, as presented in Section {\rm{\ref{sec:RN}}}.

\item The result of Theorem {\rm{\ref{th:applicationp-laplace}}} holds true under slightly weaker conditions, where we require {\rm{(H1)--(H3)}}
only on the range of $u$ and $\nabla u$ (when we have a priori estimates). 
This version will be used in Section {\rm{\ref{section:Minkowski}}}.
\end{enumerate}
\end{remark}

To prove Theorem \ref{th:applicationp-laplace} from Theorem \ref{th:abstracttheorem} we need the following lemma, which generalizes Lemma  \ref{lemma:strictconvexity}.

\begin{lemma}\label{cor:NEW}
Let $W$ be either $W^{1,p}_0(\Omega)$ or $W^{1,p}(\Omega)$ and $\Phi: W \to \mathbb{R}$ given by $\Phi(u) = \int_{\Omega} H(|\nabla u|^p) dx$ with $h$ as in {\rm{(H1)}}. Let $u, v \in W$ such that $u,v>0$ in $\Omega$ and $\gamma(t) = ((1-t)u^p+ tv^p)^{1/p}$ for $t\in [0,1]$. Then:
\begin{enumerate}[i)]
\item $t \mapsto \Phi(\gamma(t))$ is convex on $[0,1]$.
\item $t \mapsto \Phi(\gamma(t))$ is strictly convex on $[0,1]$ if $u, v \in C(\overline{\Omega}) \cap W$, $u \neq v$ and $h$ is increasing.
\item $t \mapsto \Phi(\gamma(t))$ is strictly convex on $[0,1]$ if $u$ and $v$ are linearly independent, $u, v \in C(\overline{\Omega}) \cap W$ and $h>0$ on $(0, \infty)$.
\end{enumerate} 
\end{lemma}
\begin{proof}
From (H1) we know that $H:[0,\infty) \to {\mathbb R}$ is nondecreasing and convex. Then, from \eqref{eq:pinq}, given $t_1,t_2\in [0,1]$ and  $\theta\in (0,1)$,
\begin{multline}\label{eq:NEWineqH}
\Phi(\gamma((1-\theta)t_1 + \theta t_2)) =  \int_{\Omega} H(|\nabla \gamma((1-\theta)t_1 + \theta t_2)|^p)dx \leq \int_{\Omega} H ((1-\theta) |\nabla \gamma(t_1)|^p + \theta |\nabla \gamma(t_2)|^p)\, dx \\
\leq \int_{\Omega} \left[(1-\theta)H (|\nabla \gamma(t_1) |^p) + \theta H (|\nabla \gamma(t_1)|^p) \right] dx 
= (1-\theta) \Phi(\gamma(t_1)) + \theta \Phi(\gamma(t_2)),
\end{multline}
and i) follows.

In addition to (H1), if $h$ is increasing or $h > 0$, then $H:[0,\infty) \to {\mathbb R}$ is increasing and the first inequality in 
\eqref{eq:NEWineqH} is strict if $u$ and $v$ are linearly independent 
and iii) follows. Otherwise, $u = \alpha v$ and if in addition 
$h$ is increasing, $H$ is strictly convex and the second inequality 
in \eqref{eq:NEWineqH} is strict unless $|\nabla \gamma(t_1)|=|\nabla \gamma(t_2)|$. Thus $\alpha = 1$ and $u \equiv v$ and ii) holds true. 
\end{proof}

\smallbreak
\begin{proof}[Proof of Theorem \ref{th:applicationp-laplace}]
Let $I: {W^{1,p}_0}(\Omega) \to {\mathbb R}$ be a Fr\'echet differentiable functional given by \eqref{eq:functionalsupergeneral} and $A$ be the set of positive solutions of \eqref{eq:supergeneral}. It is enough to show that conditions a)-c) of Theorem \ref{th:abstracttheorem} are satisfied.

Suppose that $u$ and $v$, with $u \neq v$, are positive solutions of \eqref{eq:supergeneral} and set $\gamma(t) = ((1-t)u^p + t v^p)^{1/p}$.
 Then, from (H4) and Corollary \ref{corollary_of_lemma:gpg}, we infer that $\gamma$ satisfies the hypotheses a) and b) of Theorem \ref{th:abstracttheorem}. From Lemma \ref{cor:NEW} i) (where we use (H1)) and (H3) we obtain $t \mapsto I(\gamma(t))$ is convex on $[0,1]$ as it is the sum of two convex functions. Therefore c) of Theorem \ref{th:abstracttheorem} holds and i) follows.

To prove ii), observe that 
by Lemma \ref{cor:NEW}, 
$t \mapsto I(\gamma(t))$ is strictly convex on $[0,1]$ if either $H$ is strictly convex (equivalent to $h$ being increasing), or the strict concavity holds at (H3). 

Finally, iii) follows from by Lemma \ref{cor:NEW} iii) if $u$ and $v$
are linearly independent, 
and hence $A \subset \{\alpha u_0; \, \alpha \in (0, \infty)\}$. 

The proof of iv) immediately follows after replacing $A$ by $A'$.
\end{proof}

\subsection{New proofs for classical equations}\label{sec:app-p}
Throughout this section $\Omega \subset {\mathbb R}^N$, with $N\geq1$, is a bounded smooth domain and $p>1$. Some of the results, namely Examples \ref{example:p-sublinear} and \ref{example:simplicity1E}, already appeared in \cite[Sections 2.5.4 and 2.5.5]{Reichel} where the author uses a different approach and the uniqueness is proved via strict variational sub-symmetry transformation groups; see \cite[Section 2]{Reichel} for more details.

\begin{example}[$p$-sublinear problems with Dirichlet boundary conditions]\label{example:p-sublinear} 
Consider \eqref{eq:supergeneral} with $h \equiv 1$ \linebreak and $g(x,t) = |t|^{q-2}t$ with $1 < q < p$, that is
\begin{equation}\label{eq:p-Lap_first}
-\Delta_p u=|u|^{q-2}u  \ \ \text{in} \ \ \Omega, \quad u = 0 \ \ \partial \Omega.
\end{equation} 
Then, by \cite{DiBenedetto, Lieberman}, any weak solution is $C^1(\overline\Omega)$.  Since $G(t) = \frac{1}{q}|t|^q$, then $G(t^{1/p}) = \frac{1}{q}|t|^{q/p}$ is strictly concave on $[0, \infty)$ as $q < p$.  Condition (H4) is satisfied by combining Hopf's Lemma \cite[Theorem 5]{Vazquez} with Lemma \ref{lemma:gradientboundarycomparison}. Hence, by Theorem {\rm{\ref{th:applicationp-laplace}}}, problem \eqref{eq:p-Lap_first} has at most one positive solution. On the other hand, it is standard to show that in this case \eqref{eq:supergeneral} has a positive solution (a global minimizer of $I$).
\end{example}

\begin{remark}\label{remark:kawohl} 
An alternative proof of the result above is presented in \cite{DiazSaa} where the path \linebreak $\gamma(t) = ((1-t)u^p + t v^p)^{1/p}$ is used to prove the integral monotonicity 
\begin{equation} \label{eq:nid}
\int_{\Omega} \left( \dfrac{-\Delta_p u^{1/p} }{u^{(p-1)/p}}+ \dfrac{\Delta_p v^{1/p}}{v^{(p-1)/p}} \right)(u-v) dx \geq 0.
\end{equation}
Alternatively, 
 it is simple to show that $\mathcal{D} = \{ v>0; v^{1/p} \in W^{1,p}_0(\Omega)\}$ is a convex cone and as proved in \cite[p. 230]{BelloniKawohl}, see also \cite{KawohlKromer},
the functional
\[
v \mapsto I(v^{1/p}) =: J(v)
\]
is strictly convex on $\mathcal{D}$. Therefore, since $I$ has a global minimizer $u \in W^{1,p}_0(\Omega)$ with $u> 0$, then $J$ has a global minimizer on $\mathcal{D}$ and the uniqueness of \emph{positive minimizers} of $I$ follows from the uniqueness of minimizer for $J$; see also \cite[Lemma A.4]{LiebSeiringerYngvason} for an application to a Gross-Pitaevskii energy functional. However, the proof of uniqueness of \emph{positive critical points} requires more attention,  as showed by our example \eqref{eq:odex}, since 
$I$ might have more critical points than $J$ and
the cone of positive functions in $W^{1,p}_0(\Omega)$ has empty interior if $1< p<N$. 
\end{remark}

\begin{example}[Simplicity of the first $p$-Laplacian eigenvalue]\label{example:simplicity1E}
The first eigenvalue of the $p$-Laplacian \linebreak operator is given by
\begin{equation}\label{eq:1stplaplacianeigenvalue}
\Lambda_p = \inf_ {u \in W^{1,p}_0(\Omega)\backslash\{0\}} \dfrac{\displaystyle{\int_{\Omega} |\nabla u|^p dx}}{\displaystyle{\int_{\Omega} |u|^p}dx} \,.
\end{equation}
Moreover, the first eigenfunctions can be characterized as the nontrivial critical points of
\[
I(u) = \frac{1}{p}\int_{\Omega} |\nabla u|^p dx - \frac{\Lambda_p}{p}\int_{\Omega}|u|^p dx, \ \ u \in W^{1,p}_0(\Omega),
\]
or as nontrivial solutions 
of
\begin{equation}\label{eq:equation-p-lapalcian-eigenvalue}
-\Delta_p  u = \Lambda_p |u|^{p-2}u  \ \ \text{in} \ \ \Omega, \quad u = 0 \ \ \partial \Omega.
\end{equation}
It is standard to show that \eqref{eq:equation-p-lapalcian-eigenvalue} has a positive solution (a global minimizer of $I$). Then the simplicity of $\Lambda_p$ is a consequence of Theorem {\rm{\ref{th:applicationp-laplace}}} iii), with $h(t) =1$, $g(x,t) = \Lambda_ p|t|^{p-2}t$ and with $A$ defined as the set of positive solutions of \eqref{eq:equation-p-lapalcian-eigenvalue}. 
Note that $h$ is not strictly increasing and $G(t^{1/p}) = \frac{\Lambda_p}{p} t$ is not strictly concave.
See \cite{deThelin, Anane, Sakaguchi, Barles, Lindqvist} for alternative and \cite{BelloniKawohl, BrascoFranzina} for similar proofs.
\end{example}

\begin{example}[Nonlinear boundary value problems]\label{example:NBC}
Consider the equation
\begin{equation}\label{eq:nonlinearboundary}
-\Delta_p u + |u|^{p-2}u =0  \ \ \text{in} \ \ \Omega, \quad |\nabla u|^{p-2}\frac{\partial u}{\partial \nu} = |u|^{q-2}u  \ \ \text{on} \ \ \partial \Omega, \quad 1 < q <p.
\end{equation}
Here we prove that \eqref{eq:nonlinearboundary} has at most one positive weak solution; see \cite[Theorem 1.2]{BonderRossi2001} for the existence of infinitely many sign-changing weak solutions.

The weak solutions of \eqref{eq:nonlinearboundary} are defined as the critical points of
\begin{equation}\label{functionalNB}
 I(u) = \frac{1}{p} \int_{\Omega} \left(|\nabla u |^p + |u|^p\right) dx - \frac{1}{q}\int_{\partial \Omega} |u|^q dS, \quad u \in W^{1,p}(\Omega).
\end{equation}
Note that since $q < p$ the boundary integral is well defined. 
By \cite{DiBenedetto, Lieberman}, any weak solutions of \eqref{eq:nonlinearboundary} is $C^1(\overline{\Omega})$. By \cite[Theorem 5]{Vazquez},
any nonnegative nontrivial critical point $v$ of \eqref{functionalNB} satisfies $v>0$ in $\overline{\Omega}$. Set 
\[
A = \{ u \in W^{1,p}(\Omega); u \ \text{is a positive weak solution of \eqref{eq:nonlinearboundary}}\}.
\]
The positivity on $\overline{\Omega}$ for any $u, v$ yields {\rm (H4)}, and from 
Remark {\rm{\ref{remark-truncation}}} with
\begin{equation}
\tilde{H} (v) = \frac{1}{p}\int_\Omega v \, dx, \qquad 
\tilde{G}(v) =  \frac{1}{p}\int_\Omega |v|^p \, dx -  \int_{\partial \Omega} |v|^q \, dS 
\end{equation}
and the strict concavity of $\tilde{G}$, we infer that $A$ has at most one element. Again it is standard to show that $A$ is not empty (contains a global minimizer of $I$).
\end{example}

\begin{example}[Nonlinear Steklov problem]
The arguments from Examples {\rm{\ref{example:simplicity1E}}} and {\rm{\ref{example:NBC}}} can be applied to prove that the first eigenvalue of 
\[
-\Delta_p u + |u|^{p-2}u =0  \ \ \text{in} \ \ \Omega, \quad |\nabla u|^{p-2}\frac{\partial u}{\partial \nu} = \lambda |u|^{p-2}u  \ \ \text{on} \ \ \partial \Omega,
\]
is simple; see \cite{MartinezRossi} for an alternative proof based on arguments from \cite[Appendix]{Lindqvist} on the strict convexity of the function 
\[
z \in {\mathbb R}^N, \quad z \mapsto |z|^p.
\]
\end{example}

\begin{example}[$p$-Laplacian Allen-Cahn problems]\label{example-Allen-Cahn} Let $q>p>1$ and consider the equation
\begin{equation}\label{eq:plaplacianAllen-Cahn}
\left\{
\begin{array}{l}
-\Delta_p u = k |u|^{p-2}u - |u|^{q-2}u \ \ \text{in} \ \Omega,\\
u> 0 \ \ \text{in} \ \ \Omega, \ \ u = 0 \ \ \text{on} \ \ \partial \Omega,
\end{array}
\right.
\end{equation}
and let $\Lambda_p$ be as in \eqref{eq:1stplaplacianeigenvalue}. Set $X = W^{1,p}_0(\Omega) \cap L^q(\Omega)$ with the norm $\|u\|_X = \| \nabla u\|_ {L^p(\Omega)} + \|u\|_{L^q(\Omega)}$, and define $I: X \to {\mathbb R}$ by
\[
I(u) = \frac{1}{p}\int_{\Omega} |\nabla u|^p \, dx  + \int_{\Omega}\left( \frac{|u|^q}{q} - k \frac{|u|^p}{p}\right) \, dx.
\]
The weak solutions of \eqref{eq:plaplacianAllen-Cahn} are defined as the nontrival nonnegative critical points of $I$. By testing \eqref{eq:plaplacianAllen-Cahn} with
$u$ we infer that \eqref{eq:plaplacianAllen-Cahn} has no weak solution if $k \leq \Lambda_p$. So we consider $k > \Lambda_p$. By testing \eqref{eq:plaplacianAllen-Cahn} with $(u - k^{1/q-p})^+$ we can show that any nonnegative solution satisfies $\|u\|_ {L^{\infty}(\Omega)}\leq k^{1/(q-p)}$ and by \cite{DiBenedetto, Lieberman},  $u \in C^1(\overline\Omega)$. Then the Hopf Lemma, as in \cite[Theorem 5]{Vazquez}, combined with Lemma {\rm{\ref{lemma:gradientboundarycomparison}}} guarantees {\rm{(H4)}}. Then, from Theorem {\rm{\ref{th:applicationp-laplace}}}, by setting 
\[
A = \{ u \in X; u > 0 \ \text{is a weak solution of \eqref{eq:plaplacianAllen-Cahn}}\}
\]
and noting that $h \equiv 1$ and $G (t) = \frac{k}{p}|t|^p - \frac{1}{q}|t|^q$ with $G (t^{1/p})$ being strictly concave on $[0,\infty)$, we infer that  \eqref{eq:plaplacianAllen-Cahn} has at most one (positive) weak solution (Corollary {\rm{\ref{corollary_of_lemma:gpg}}} guarantees that $\gamma(t) = ((1-t)u^p - t v^p)^{1/p}$ is locally Lipschitz at $t =0$). Finally, with $k > \Lambda_p$, it is simple to show that $A$ contains a global minimizer of $I$. See \cite[Theorem 4]{Berestycki1981} and \cite[Theorem 6]{Berger1979} for alternative proofs in the case of $p=2$ and $q=3$. For the general case $q>p>1$ an alternative proof follows using the arguments in \cite{DiazSaa} based on the relation \eqref{eq:nid}.
\end{example}

\section{Mean curvature type operators}\label{section:meancurvatureoperators}

\subsection{Mean curvature operator in Euclidean space}
In this section we investigate solutions of 
\begin{equation}\label{eq:crvt}
-\operatorname{div}\left( \frac{\nabla u}{\sqrt{1+|\nabla u|^2}}\right)= g(x, u) \quad \text{in }\Omega,\qquad u=0 \quad \text{on }\partial\Omega,
 \end{equation}
where $\Omega \subset {\mathbb R}^N$, with $N\geq1$, is a bounded smooth domain. Note that problem \eqref{eq:crvt} has the structure of \eqref{eq:supergeneral} with $h(t)=(1+t)^{-\frac{1}{2}}$, but $h$ does not satisfy $(H1)$ since it is a decreasing function. A model nonlinearity in this section is $g(x, u) = \lambda u^{p-1}$ with $p \in (1, 2)$, $\lambda > 0$; however, in this case it is known \cite{Le2005, HabetsOmari} that there are multiple non-negative solutions of \eqref{eq:crvt}. Even in this case, our method yields the 
uniqueness of solutions in certain subsets of the state space, specifically for functions with an additional bound on the gradient. 
We remark that the existence of small $C^1$-solutions was proved in the one-dimensional case in 
\cite{HabetsOmari, BonheureHabets} and the existence of small solutions in higher dimensions in \cite{ObersnelOmari2010}. Our uniqueness 
results provide new insights on the bifurcation diagrams obtained in \cite[Fig. 2]{BonheureHabets} and \cite[Fig. 1]{HabetsOmari}. Observe that our results also apply to Allen-Cahn-type nonlinearities like $g(x,u)=k |u|^{p-2}u- |u|^{q-2}u$ with $k>0$, $q>p$ and $p\in (1,2)$.

\begin{theorem}
If there exists $p \in (1, 2)$ such that the function $t \mapsto G(x, t^{\frac{1}{p}})$ is concave and {\rm{(H2)}} (from Section {\rm{\ref{sec:application2order}}}) is satisfied, then there exists at most one positive solution of 
\eqref{eq:crvt} in the set 
\[
Z := \left\{u \in W^{1, \infty}_0(\Omega); \|\nabla u\|_{L^\infty(\Omega)}  < \left(\frac{2 - p}{p - 1}\right)^{1/2}\right\}.
\]
\end{theorem}

\begin{proof}
We verify assumptions of Theorem \ref{th:abstracttheorem} with 
the curve $\gamma$ defined for any $u, v \in Z$, $u,v>0$ by
\begin{equation}
\gamma(t) = ((1-t)u^p + t v^p)^{\frac{1}{p}} \qquad t \in [0, 1] \,.
\end{equation}
Note that solutions to \eqref{eq:crvt} that belong to $Z$ are critical points of  the Fr\'echet differentiable functional
\begin{align}
I: W^{1,p}_0(\Omega)\to \mathbb R,\qquad I(u):= \int_{\Omega} \sqrt{1+|\nabla u|^2} - G(x,u) \, dx \,.
\end{align}
To prove the convexity of $t \mapsto I(\gamma(t))$ we use Lemma 
\ref{lem:gin} with 
\begin{equation}
Q (z) = z^p \qquad z \in (0, \infty), \quad \textrm{and } M(z) = \sqrt{1 + z^2} \qquad z \in \left[0,  \left(\frac{2 - p}{p - 1}\right)^{1/2}\right)\,.
\end{equation}
Then $F_1 (z_1) = pz_1^{(p - 1)/p}$ for any $z_1 \in (0, \infty)$ and $F_2(z_2) = \sqrt{z_2^2 - 1}$ for any 
$z_2 = M(z) \in [1, 1/\sqrt{p-1})$.
It is easy to check that $F_2''(z_2)< 0$, $F_1''(z_1) < 0$ for any $(z_1, z_2) \in (0, \infty) \times (1, 1/\sqrt{p-1})$. 
The strict concavity of $F$ on $(0, \infty) 
\times (1, 1/\sqrt{p-1})$ now follows by \eqref{eq:opp} from 
\begin{equation}
\left(\left(\frac{F_1}{F_1'}\right)' - 1 \right) \left(\left(\frac{F_2}{F_2'}\right)' - 1 \right) = 
  \frac{1}{p - 1} \left(  \frac{1}{z_2^2} \right) > 1\,.
\end{equation}
Note that this is the only step where we need a restriction on the gradient. 
Thus from Lemma \ref{lem:gin} 
we have that $t \mapsto M(|\nabla \gamma(t)|)$ is strictly convex
whenever at least one of $|\nabla u|$, $|\nabla v|$ is positive. Here and below, the gradient of a function is understood in a weak sense. 
 Clearly, $M(|\nabla \gamma(\cdot)|) \equiv 1$ if $|\nabla u| = |\nabla v| =0$. If $|\nabla u| = 0$ almost everywhere, then $u$ is constant, and therefore zero by the boundary conditions, a contradiction to $u > 0$. 
 Thus $|\nabla u| > 0$ on  a set of positive measure, and on that set
 $t \mapsto M(|\nabla \gamma(t)|)$ is strictly convex, and consequently 
 for each $u, v > 0$
\begin{equation}
t \mapsto \int_{\Omega} M(|\nabla \gamma(t)|) \, dx 
\qquad \textrm{ is strictly convex in }[0,1]\,.
\end{equation}
Since $t \mapsto G(x, \gamma(t))$ is concave, the strict convexity of $t \mapsto I(\gamma(t))$ follows. 

Moreover, any solution in $Z$ of \eqref{eq:crvt}  
satisfies a uniformly elliptic equation, and consequently it is smooth 
by elliptic regularity, 
and moreover satisfies the maximum principle and Hopf lemma.  Thus by Lemma \ref{lemma:gradientboundarycomparison}
and Corollary \ref{corollary_of_lemma:gpg} we obtain condition Theorem 
\ref{th:abstracttheorem} b) , and the uniqueness follows. 
\end{proof}

\subsection{Mean curvature operator in Minkowski space}\label{section:Minkowski}
We now consider a quasilinear Dirichlet problem involving the mean curvature operator in Minkowski space, namely
\begin{equation}\label{ctp}
-\operatorname{div}\left( \frac{\nabla u}{\sqrt{1-|\nabla u|^2}}\right)= g(x, u) \quad \text{in }\Omega,\qquad u=0 \quad \text{on }\partial\Omega,
 \end{equation}
where $\Omega \subset {\mathbb R}^N$, with $N\geq1$, is a bounded domain. Existence of positive solutions can be found by minimization
of 
\begin{align}\label{I:sthg}
I(u):= \int_{\Omega} \left(1-\sqrt{1-|\nabla u|^2}\right) - G(x,u) dx \,,
\end{align}
where $G(x, t) := \int_0^{t}g(x,s)\ ds$ on the convex set 
\begin{align}
K_0:= \{u \in W^{1,\infty}(\Omega); |\nabla u| \leq 1, \ u=0 \text{ on }\partial \Omega\},\label{K:0}
\end{align}
under suitable assumptions on $g$, see for example \cite{Mawhin}. 

Set
\begin{align*}
\operatorname{diam}(\Omega):=\sup\{|x-y|; x,y\in\overline{\Omega}\},\quad M:=\frac{\operatorname{diam}(\Omega)}{2},
\end{align*}
and   we have that 
\begin{align}
&\|u\|_{L^\infty(\Omega)}\leq M\quad\text { for all }u\in K_0 \,. \label{unif:bound:1}
\end{align}

Our main contribution  is the following new uniqueness result.

\begin{theorem}\label{ctp:thm}
Let $\Omega\subset\mathbb R^N$ be a bounded smooth domain and assume that
\begin{itemize}
\item [$(G)$] for every $x\in \Omega$, the function $t\mapsto G(x,\sqrt{t})$ is concave in $[0, M^2]$,
\item [$(g)$] The function $g:\overline{\Omega}\times [0,M]\to\mathbb R$ is continuous and of the form $g=g_1+g_2$, with $g_i(x,0)=0$ for $i=1,2$,  where $t\mapsto g_1(x,t)$ is Lipschitz continuous in $[0,M]$ uniformly in $x$ and $g_2$ is continuous and nonnegative in $\overline{\Omega}\times [0,M]$. 
\end{itemize}
Then \eqref{ctp} has at most one positive classical solution, that is, the set
\begin{equation*}
{\mathcal A}:=\{u\in C^{2,\alpha}(\overline{\Omega}); u>0 \text{ in }\Omega \text{ and }u \text{ satisfies \eqref{ctp}}\}
\end{equation*}
contains at most one element.
\end{theorem}

\begin{remark}
 \begin{enumerate}[i)]
 \item $(g)$ is used to show {\rm{(H4)}} from Section \ref{sec:application2order}.
 \item If one assumes $(g)$ and that $g(x,\cdot)$ is \emph{nonincreasing} in $[0,M]$ for every $x\in\Omega$, then \eqref{ctp} has only one solution by the convexity of the energy functional {\rm{(}}see \cite[Proposition 1.1]{BartnikSimon}{\rm{)}}.
 \item If the concavity assumption $(G)$ is dropped, then there are results on multiplicity of nontrivial nonnegative solutions. In particular, \cite[Theorem 3]{Mawhin} shows that \eqref{ctp} has at least two nontrivial nonnegative solutions if $g(x,u)= k |u|^{p-2}u-u$, $p>2$, and $k>0$ is large enough.
 \end{enumerate}
\end{remark}

Let $\lambda_1$ denote the first Dirichlet eigenvalue of the Laplacian in $\Omega$.  Theorem \ref{ctp:thm} and standard existence arguments imply the following result which applies in particular to $g(x,u) = k u - |u|^{p-2}u$ with $k>2\lambda_1,$ $p>2$  or to $g(x,u)=|u|^{q-2}u$ for $q\in(1,2).$

\begin{corollary}\label{coro:ACuP}
Let $\Omega \subset \mathbb{R}^N$, $N\geq 1$, be a bounded smooth domain. In addition to $(G)$, $(g)$ from Theorem \ref{ctp:thm}, assume that $g$ is H\"{o}lder continuous in $\overline{\Omega}\times [0,M]$, $g(\cdot,0)=0$ in $\Omega$, and 
\begin{equation}\label{asgl}
\frac{G(x, t)}{t^2} > \lambda_1 \, \qquad \text{as} \ \ t \to 0^+, \ \ \text{uniformly in} \ \  x\in {\Omega}.
\end{equation}
Then \eqref{ctp} has a unique positive classical solution.
\end{corollary}

We prove first the main theorem of this section.

\begin{proof}[Proof of Theorem {\rm{\ref{ctp:thm}}}] 
By \eqref{unif:bound:1} and \cite[Corollary 3.4 and Theorem 3.5]{BartnikSimon} there is $\theta\in (0,1)$ such that
\begin{align}
&{\mathcal A}\subset \{ u\in {K_0}; |\nabla u|\leq 1-\theta\}.\label{unif:bound:2}
\end{align}
Let
\begin{equation}
 \hat g(x,t):=\textrm{sign} (t) g(x,\min\{|t|,M\})   \qquad \text{ for }(x,t)\in \overline{\Omega}\times \mathbb R.
\end{equation}
and $\hat G(x,t):=\int_0^t \hat g(x,s)\, ds$. Note that $\hat g$ is just a truncation of $g$ extended as an odd function. Moreover, let $\tilde h\in C([0,\infty))$ be given by 
\[
\tilde h(t)=\left\{
\begin{array}{ll}
\frac{1}{\sqrt{1-t}}, \qquad &t\in[0, 1-\theta]\vspace{5pt},\\
\frac{1}{\sqrt{\theta}}, \qquad &t \geq 1-\theta,
\end{array}
\right.
\]
with $\theta$ as in \eqref{unif:bound:2} and let $\tilde H(t) := \int_0^t \tilde h(s) ds$ for $t\geq 0$.  By \eqref{unif:bound:1} and \eqref{unif:bound:2}, the elements of $\mathcal A$ are critical points of 
\begin{align}\label{I:tilde}
\tilde I:H^{1}_0(\Omega) \to \mathbb R,\qquad \tilde I(u):= \int_{\Omega} \frac{\tilde H(|\nabla u|^2)}{2} - \hat G(x,u) dx.
\end{align}
Clearly $\tilde h$ satisfies assumption $(H1)$, $\hat g$ satisfies $(H2)$, and, by $(G)$, for every $x \in \Omega$, 
$t \mapsto \hat G(x, \sqrt t) = G(x,\sqrt t)$ is concave on $[0, M^2]$ and so (H3) is satisfied. We now show that $(H4)$ holds for all elements in $\mathcal A$. Note that these might not be 
all the critical points of $\tilde I$. We argue as in \cite[Lemma 2.2]{Corsato}. Let $u\in \mathcal A$ and note that $u$ solves the uniformly elliptic linear equation 
$-\sum_{i,j=1}^N a_{ij}\partial_{x_ix_j} u-c\, u = \rho$, where  
\begin{equation}\label{aij}
a_{ij}:=a(|\nabla u|^2)\delta_{ij}+2a'(|\nabla u|^2)\partial_{x_i}u\partial_{x_j}u\quad \text{ with }a(s):=(1-s)^{-\frac{1}{2}},
\end{equation}
$\delta_{ij}$ is the Kronecker delta, $c(x):=\int_0^1g_1'(x,s u(x))\, ds$, and $\rho(x):=g_2(x,u(x))$  for a.e. $x\in\Omega$. By $(g)$ we have that $\rho\geq 0$ and $c\in L^\infty({\Omega})$. By Hopf's Lemma we obtain that $\frac{\partial u}{ \partial \eta} < 0$ on $\partial \Omega$. Therefore, for every $u,v\in\mathcal A$ there is $\delta>1$ such that $\frac{1}{\delta} v \leq u \leq \delta v$.
The result now follows from  Theorem \ref{th:applicationp-laplace}, iv) and ii). 
\end{proof}

\begin{proof}[Proof of Corollary {\rm{\ref{coro:ACuP}}}]
The uniqueness follows from Theorem \ref{ctp:thm}. We now show that 
\[
{\mathcal A}:=\{u\in C^{2,\alpha}(\overline{\Omega}); u>0 \text{ in }\Omega \text{ and }u \text{ satisfies \eqref{ctp}}\}
\] 
is nonempty. Here and below $\alpha\in(0,1)$ denotes possibly different H\"{o}lder exponents. Let
\begin{equation}\label{I:def}
\hat I:K_0\to \mathbb R;\qquad \hat I(u) = \int_{\Omega} \left(1-\sqrt{1-|\nabla u|^2}\right) - \hat G(x,u) \, dx,
\end{equation}
with $\hat G$ and $\hat g$ as in the proof of Theorem \ref{ctp:thm}.  Since $\hat g$ is continuous and bounded in $\overline{\Omega}\times \mathbb R$, \cite[Theorem 1]{Mawhin} implies that $\hat I$ attains its infimum at $\bar u\in {K_0}\cap W^{2,p}(\Omega)$ for some $p>N$ with $\|\nabla \bar u\|_{L^\infty(\Omega)}<1$ and
$\bar{u}$  solves \eqref{ctp} with $g$ replaced by $\hat g$.  By \eqref{asgl} and the inequality $1 - \sqrt{1-t} \leq t$ for $t\in[0,1]$, one has $\hat I(\varepsilon \varphi_1) < 0$ for $\varepsilon > 0$ small enough, where $\varphi_1$ is the positive eigenfunction associated to $\lambda_1$ normalized in $L^2(\Omega)$. Therefore, $\hat I(\bar{u}) < 0$ and hence $\bar{u} \not \equiv 0$. Moreover, since $t\mapsto \hat G(x,t)$ is even by construction, we have that $\hat I(|\bar u|)=\hat I(\bar u)$, 
and consequently $\bar v:= |\bar u|\in {K_0}$ is also a critical point of $\hat I$ (in the sense of \cite{Mawhin}). Therefore \cite[Theorem 1]{Mawhin} implies that $\bar v\in {K_0}\cap W^{2,p}(\Omega)$ for some $p>N$ with $\|\nabla \bar v\|_{L^\infty(\Omega)}<1$ and solves \eqref{ctp}, since $\hat g(x,\bar v)=g(x,\bar v)$ 
 by \eqref{unif:bound:1}.

Arguing as in the proof of Theorem {\rm{\ref{ctp:thm}}},
$-\sum_{i,j=1}^N a_{ij}\partial_{x_ix_j} \bar v= \eta$ with $a_{ij}$ as in \eqref{aij} and $\eta(x):=g(x,\bar v(x))$.  Since $\eta \in C^{0,\alpha}(\overline{\Omega})$ and $\bar v\in C^{1,\alpha}(\overline{\Omega})$ by Sobolev embeddings, we have that $a_{ij}\in C^{0,\alpha}(\overline{\Omega})$. Then \cite[Theorem 6.14]{GilbargTrudinger} yields that $\bar v\in C^{2,\alpha}(\overline{\Omega})$.  Finally, $\bar v>0$ in $\Omega$ by the maximum principle. 
\end{proof}

\section{Problems involving fractional Laplacians}\label{section:fractionalLaplacian}

In this section we use Theorem \ref{th:abstracttheorem} to obtain new uniqueness results for nonlocal operators.

\medskip
Let $s\in(0,1)$ and $\Omega\subset\mathbb R^N$, $N\geq 1$ be a bounded smooth domain and let
\begin{align*}
{\mathcal H}_0^s(\Omega):=\left\{\, u\in L^2({\mathbb R}^N); \|u\|_{H^s}:=\left(\int_{\mathbb R^N}\int_{\mathbb R^N} \frac{|u(x)-u(y)|^2}{|x-y|^{N+2s}}\, dxdy\right)^{\frac{1}{2}}<\infty\quad \text{and}\quad u\equiv 0\text{ in }\mathbb {\mathbb R}^N\backslash \Omega\, \right\}
\end{align*}
denote the \emph{fractional Sobolev space of order $s$} and
\begin{align*}
 (-\Delta)^s u(x) = P.V. \int_{\mathbb R^N} \frac{u(x)-u(y)}{|x-y|^{N+2s}}\, dy = \lim_{\varepsilon\to 0}\int_{|x-y|\geq\varepsilon} \frac{u(x)-u(y)}{|x-y|^{N+2s}}\, dy 
\end{align*}
denote the \emph{fractional Laplacian of order $s$}, where we have omitted any normalization constant for simplicity.  We say that $u\in C^{0,s}(\overline{\Omega})\cap {\mathcal H}^s_0(\Omega)$ is a weak solution of 
\begin{equation}\label{frac:eq}
(-\Delta)^s u= g(x,u) \quad \text{in }\Omega,\qquad u\equiv 0 \quad \text{in }\mathbb R^N\backslash\Omega\,,
 \end{equation}
 if
\begin{equation}\label{weak:frac:sol}
 \frac{1}{2}\int_{\mathbb R^N}\int_{\mathbb R^N} \frac{(u(x)-u(y))(\phi(x)-\phi(y))}{|x-y|^{N+2s}}\, dxdy=\int_{\Omega} g(x,u(x))\phi(x) \, dx
  \end{equation}
for all $\phi\in {\mathcal H}^s_0(\Omega)$. 
 Let $X$ be a subspace of ${\mathcal H}_0^s(\Omega)$ such that the energy associated to \eqref{frac:eq}
\begin{align}\label{frac:ene}
I(u):=\frac{1}{4}\int_{\mathbb R^N}\int_{\mathbb R^N} \frac{|u(x)-u(y)|^2}{|x-y|^{N+2s}}\, dxdy -
\int_\Omega G(x, u(x)) \ dx \qquad \textrm{for all } u \in X
\end{align}
is well defined and Fr\'{e}chet differentiable, where $G$ is given by \eqref{eq:gdf} and
\begin{align}\label{G:frac}
\text{ the derivative of $u \mapsto \int_{\Omega} G(x,u)\,dx$ evaluated at $v\in X$ is $\int_{\Omega} g(x,u) v\, dx$.}
\end{align}
The choice of $X$ depends on the nonlinearity and for most applications one can consider $X=\mathcal{H}^s_0(\Omega)$ or $X=\mathcal{H}^s_0(\Omega)\cap L^p(\Omega)$ for some $p>2$. In the latter case we endow $X$ with the norm $\|u\|_X = \|u\|_{H^s} + \|u\|_{L^p(\Omega)}$.

\begin{theorem}\label{th:fractional}
Let $\Omega\subset\mathbb R^N$ be a bounded smooth domain, $s\in(0,1)$, and assume that
\begin{itemize}
\item [$(G)$] for every $x\in \Omega$, the function $t\mapsto G(x,\sqrt{t})$ is strictly concave in $[0, \infty)$,
\item [$(g)$] The function $g:\overline{\Omega} \times \mathbb R \to\mathbb R$ is continuous of the form $g=g_1+g_2$, with $g_1(x,0)=0$ and  $t\mapsto g_1(x,t)$ is locally Lipschitz continuous in $\mathbb R$ uniformly in $x$, and $g_2$ is continuous and nonnegative in $\overline{\Omega}\times [0, \infty)$. 
\end{itemize}
Let  $X=\mathcal{H}^s_0(\Omega)$ or $X=\mathcal{H}^s_0(\Omega)\cap L^p(\Omega)$ for some $p>2$ such that $I$ as in \eqref{frac:ene} is Fr\'{e}chet differentiable in $X$ and \eqref{G:frac} is satisfied. Then \eqref{frac:eq} has at most one positive weak solution $u\in C^{0,s}(\overline{\Omega})\cap {\mathcal H}^s_0(\Omega)$.
\end{theorem}
\begin{proof}
Let $X=\mathcal{H}^s_0(\Omega)\cap L^p(\Omega)$ endowed with the norm $\|u\|_X = \|u\|_{H^s} + \|u\|_{L^p(\Omega)}$ (the case $X=\mathcal{H}^s_0(\Omega)$ follows similarly). We use Theorem \ref{th:abstracttheorem} with $A\subset C^{0,s}(\overline{\Omega})\cap X$ being the set of nontrivial nonnegative weak solutions of \eqref{frac:eq}. Observe that, by the maximum principle (see, for example, \cite{jarohs:hopf,jarohs:phd}), any element of $A$ is positive in $\Omega$. Fix $u,v\in A$. For $t\in[0,1]$, let $\gamma(t):=((1-t)u^2 + tv^2)^{\frac{1}{2}}$.  We show first that $\gamma:[0,1] \to X$ is well defined and Lipschitz at $t=0$. We start with the $H^s$-norm, that is, we show that
\begin{align}\label{frac:Lip:zero}
 \left\|\frac{\gamma(t)-\gamma(0)}{t}\right\|_{H^s}\leq C(\|u\|_{H^s}+\|v\|_{H^s})\qquad \text{ for }t\in(0,1],
\end{align}
and for some $C>0$. Indeed, note that
\begin{align*}
 \frac{\gamma(t)-\gamma(0)}{t}=\frac{v^2-u^2}{\gamma(t)+u}=u\frac{w^2-1}{(1-t+tw^2)^{\frac{1}{2}}+1}=u\, z(w,t)\,,
\end{align*}
where $w:=\frac{v}{u}\chi_{\Omega}$, $z(\xi,t):=\frac{\xi^2-1}{(1-t+t\xi^2)^{\frac{1}{2}}+1}$, and $\chi_{\Omega}$ is the characteristic function of $\Omega$. Since $u,v\in C^{0,s}(\overline{\Omega})$, there is $M>0$ such that $0<v<M\delta^s$ in ${\Omega}$, where $\delta(x):=\operatorname{dist}(x,\mathbb R^N\backslash\Omega)$. Moreover, $u$ is a weak solution of $(-\Delta)^s u-c(x)u=\rho$, with $c(x):=\int_0^1g_1'(x,s u(x))\, ds$, and $\rho(x):=g_2(x,u(x))$  for a.e. $x\in\Omega$. By $(g)$ we have that $\rho\geq 0$ and $c\in L^\infty({\Omega})$. Then, by Hopf's Lemma (see \cite{jarohs:hopf}) there is $m>0$ such that $u>m\delta^s$ in ${\Omega}$. Therefore $w<\frac{M}{m}$ in $\Omega$. Moreover, for $x,y\in\Omega$, 
\begin{align*}
u(x)z(w(x),t)-u(y)z(w(y),t)=&(u(x)-u(y))z(w(x),t)-u(y)(z(w(y),t)-z(w(x),t))
\end{align*}
and by the Mean value Theorem
\begin{align*}
 u(y)|z(w(x),t)-z(w(y),t)|&\leq C_1u(y)|w(x)-w(y)|= C_1|v(x)-v(y)+\frac{v(x)}{u(x)}(u(y)-u(x))|
 \\&\leq C_2(|v(x)-v(y)|+|u(y)-u(x)|)\,,
\end{align*}
where $C_1:=\sup\limits_{k\in[0,\frac{M}{m}] ,t\in[0,1]}|\partial_w z(k,t)|<\infty$ and $C_2:=C_1+\frac{M}{m}$.  

On the other hand, if $x\in\Omega$ and $y\in\mathbb R^N\backslash \Omega$, then $u(y) = 0$ and
\[
|u(x)z(w(x),t)-u(y)z(w(y),t)|\leq C_ 1\|w\|_{L^\infty(\Omega)}|u(x)-u(y)|\,.
\]
These estimates readily imply \eqref{frac:Lip:zero} and therefore $\gamma$ is Lipschitz at 0 with respect to the $H^s$-norm. Moreover, by Corollary \ref{corollary_of_lemma:gpg} we have that $\gamma$ is Lipschitz at $t=0$ with respect to the $L^p$-norm, and therefore $\gamma$ is Lipschitz at 0 with respect to the norm in $X$.

Finally, to prove the strict convexity of $t\mapsto I(\gamma(t))$ it suffices (see \eqref{eq:globalargument} in the proof of Lemma \ref{lem:gin}) to show that
\begin{equation}\label{eq:igm}
(\gamma(t)(x)-\gamma(t)(y))^2\leq (1-t)(u(x)-u(y))^2+t(v(x)-v(y))^2 \qquad \text{ for }x, y \in \mathbb{R}^N \,
\end{equation}
(recall that $t\mapsto -\int_\Omega G(x, \gamma(t)(x)) \ dx$ is strictly convex by assumption $(G)$).  Indeed, after the substitution $a = \sqrt{1 - t}\,u(x)$, $b = \sqrt{1 - t}\,u(y)$, $c = \sqrt{t}\,v(x)$, and $d = \sqrt{t}\,v(y)$ this is equivalent to 
\begin{equation}
\Big|(a^2 + c^2)^{\frac{1}{2}} - (b^2 + d^2)^{\frac{1}{2}}\Big| \leq \Big((a - b)^2 + (c - d)^2\Big)^{\frac{1}{2}} \,,
\end{equation}
which follows from the Minkowski inequality.  Thus all the assumptions from Theorem \ref{th:abstracttheorem} are satisfied and the uniqueness follows. 
\end{proof}

\begin{remark}
Note that whenever (H4) from Section \ref{sec:application2order} holds, we can 
use the path $\gamma(t):=((1-t)u^p + tv^p)^{\frac{1}{p}}$, $p \geq 1$ to show
the uniqueness of a nontrivial nonnegative critical point of the functional
\begin{equation}
I(u) = \frac{1}{2p} \int_{\mathbb{R}^N}\int_{\mathbb{R}^N} \frac{|u(x)-u(y)|^{p}}{|x-y|^{N+sp}}\, dx \, dy - 
\int_\Omega G(x, u(x)) \ dx \qquad \textrm{for all } u \in W^{s, p}_0(\Omega) \,,
\end{equation}
where $t \mapsto G(x, t^{1/p})$ is strictly concave, cf. (H3). This functional is related to the fractional
$p$-Laplacian $(-\Delta)_p^s$, $s \in (0, 1)$. Assumption (H4) can be deduced from a Hopf-type lemma, but, to our best knowledge, such result 
is established only if $p = 2$. 

\end{remark}

Let $\Lambda_1>0$ and $\Phi_1\in C^{0,s}(\overline{\Omega})\cap {\mathcal H}^s_0(\Omega)$ be such that 
\begin{equation}\label{1st:frac:eigenf}
(-\Delta)^s \Phi_1 = \Lambda_1 \Phi_1 \quad \text{ in }\Omega,\qquad \Phi_1 > 0 \quad \text{ in }\Omega, \qquad \Phi_1 \equiv 0 \quad \text{ in }\mathbb R^N\backslash\Omega,
\end{equation}
see, for example, \cite[Theorem 5.23]{jarohs:phd}.  Then Theorem \ref{th:fractional} and standard minimization arguments imply the following.

\begin{corollary}\label{cor:frac}
Let $\Omega\subset\mathbb R^N$, $N\geq 1$ be a bounded domain of class $C^{2}$, $s\in(0,1)$, $p>2$, and $k>{\Lambda_1}$. Then \eqref{frac:eq} with $g(x,u)=ku-|u|^{p-2}u$ has a unique positive weak solution $u\in C^{0,s}(\overline{\Omega})\cap {\mathcal H}^s_0(\Omega)$.
\end{corollary}
\begin{proof}
The uniqueness follows from Theorem \ref{th:fractional}. The existence follows from a standard global minimization argument. We recall that $X := \mathcal{H}^s_0(\Omega)\cap L^p(\Omega)$ is (compactly) embedded into $L^2(\Omega)$; see, for example, \cite[Theorem 7.1]{Hitchhiker}. It is standard to see that $I$ is bounded from below and 
weakly lower semi-continuous, and therefore a global minimizer $u\in X$ is attained. Moreover,  $u$ is nontrivial because the condition $k > {\Lambda_1}$ guarantees that $I(\varepsilon\Phi_1)=\int_\Omega(\frac{\Lambda_1}{2}-\frac{k}{2}+ \frac{\varepsilon^{p-2}}{p}\Phi_1^{p-2})\varepsilon^2\Phi_1^2\ dx<0$ for $\epsilon \sim 0$, by \eqref{1st:frac:eigenf}. Finally, since $|u|\in X$ and $I(|u|)\leq I(u)$, we may assume that $u\geq 0$ in $\mathbb R^N$.  

It only remains to show that $u>0$ in $\Omega$ and $u\in C^s(\overline{\Omega})$.  Let $\phi=(k^{1/(p-2)}-u)^-\in X$. Then \eqref{weak:frac:sol}, 
the fact that the positive and the negative part of a function have disjoint supports, and  $u(k-|u|^{p-2})\phi \leq 0$ 
yield
\begin{align*}
 0&=\frac{1}{2}\int_{\mathbb R^N}\int_{\mathbb R^N} \frac{(u(x)-u(y))(\phi(x)-\phi(y))}{|x-y|^{N+2s}}\, dxdy-\int_{\Omega} u(x)(k-|u(x)|^{p-2})\phi(x)\, dx\\
 &\geq \frac{1}{2}\int_{\mathbb R^N}\int_{\mathbb R^N} \frac{(\phi(x)-\phi(y))^2}{|x-y|^{N+2s}}\, dxdy-\int_{\Omega} u(x)(k-|u(x)|^{p-2})\phi(x) \, dx\geq 0\,.
  \end{align*}
 Therefore $\|u\|_{L^\infty(\mathbb R^N)}\leq k^{1/(p-2)}$. This implies that the right-hand side of \eqref{frac:eq} is nonnegative and bounded. It follows by regularity, see \cite{regularity:frac}, that $u\in C^{0,s}(\overline{\Omega})$.  The strict positivity $u>0$ in $\Omega$ follows from the maximum principle; see, for example, \cite{jarohs:hopf,jarohs:phd}. This ends the proof.
\end{proof}

\section{Hamiltonian elliptic systems}\label{sec-HS}

Let $\Omega \subset {\mathbb R}^N$, $N\geq 1$, be a bounded smooth domain and consider the Hamiltonian elliptic system
\begin{equation}\label{eq:more general system}
\left\{
\begin{array}{l}
 - \Delta u = |v|^{q-1}v \quad \text{in} \quad \Omega,\\
 -\Delta v = |u|^{p-1}u \quad \text{in} \quad \Omega,\\
 u, v = 0 \quad \text{on} \quad \partial \Omega
 \end{array}
 \right.
\end{equation}
in the sublinear case, a notion which, for these systems, corresponds to
\begin{equation}\label{eq:hipothesissublinear}
p, q> 0 \quad \text{and} \quad p\cdot q < 1 \,.
\end{equation}
The uniqueness of a positive solution to \eqref{eq:more general system} was proved in \cite[Theorem 3]{Dalmasso2000}, see also \cite{montenegro}, using Krasnoselskii \cite{Krasnoselskii} type argument. Here we present an alternative proof based on Theorem \ref{th:abstracttheorem}.

\begin{theorem}\label{th:hamiltoniansystem}
If \eqref{eq:hipothesissublinear} is satisfied, then \eqref{eq:more general system} has a unique positive classical solution. 
\end{theorem}

\smallbreak

To prove Theorem \ref{th:hamiltoniansystem} we will treat \eqref{eq:more general system} by using the dual variational method. We mention that this approach has been used in \cite{ClementvanderVorst} in the case  $p> 1$ and $q>1$; see also \cite[Section 3]{BonheuredosSantosTavares} for general $p$ and $q$.
\smallbreak

For any $r > 0$ define $\phi_r: {\mathbb R} \rightarrow {\mathbb R}$ by
\[
\phi_r(t) = |t|^{r-1}t,  \quad t \in {\mathbb R}.
\]
Then we rewrite \eqref{eq:more general system} as
\[
u=(-\Delta)^{-1}(|v|^{q-1}v)=(-\Delta)^{-1}(\phi_q(v)),\quad v=(-\Delta)^{-1}(|u|^{p-1}u)=(-\Delta)^{-1}(\phi_p(u))
\]
and after introducing the new variables $f=|u|^{p-1}u=\phi_p(u)=-\Delta v$, $g=|v|^{q-1}v=\phi_q(v)=-\Delta u$, we are led to the system
\begin{equation}\label{euler-lagrange}
(-\Delta)^{-1}f = |g|^{\frac{1}{q}-1}g=\phi_q^{-1}(g), \qquad (-\Delta)^{-1}g = |f|^{\frac{1}{p}-1}f=\phi_p^{-1}(f).
\end{equation}

Define $K := (-\Delta)^{-1}$,  $X:=L^{\frac{p+1}{p}}(\Omega)\times L^{\frac{q+1}{q}}(\Omega)$, and since $\int_\Omega f K g \, dx = \int_{\Omega} g K f \, dx$, it follows that the system \eqref{euler-lagrange} appears as the Euler-Lagrange equations associated to the action functional
\begin{equation}\label{phi}
\Phi(f,g) = \int_\Omega \left(\frac{p}{p+1}|f|^{\frac{p+1}{p}} + \frac{q}{q+1}|g|^{\frac{q+1}{q}}\right) dx -\int_{\Omega} f K g \, dx, \quad (f,g) \in X.
\end{equation}
It is well known that $(f,g) \in X$ is a critical point of $\Phi$ if, and only if, $(u,v) = (Kg, Kf) = ( \phi_ p^{-1}(f), \phi_ q^{-1}(g))$ is a classical solution of \eqref{eq:more general system}. 

Within this framework, with $\Phi$ as defined above, we prove  Theorem \ref{th:hamiltoniansystem} with help of Theorem \ref{th:abstracttheorem}. We need same preliminaries results.

\begin{theorem}[Generalized Minkowski inequality]\label{t:mkw}
Let $F : X \times Y \to {\mathbb R}$ be a measurable function on a $\sigma$-finite measure space $X \times Y$ and let $\mu$, $\nu$ be the respective measures on $Y$ and $X$. Then, for any $1 \leq p < \infty$,
\begin{equation}
\left(\int_{X} \left( \int_Y |F| d\mu(\eta) \right)^{p} d\nu(\xi) \right)^{\frac{1}{p}} \leq 
\int_{Y} \left( \int_X |F|^p d\nu(\xi) \right)^{\frac{1}{p}} d\mu(\eta) \,.
\end{equation}
\end{theorem}
For the proof of the above Generalized Minkowski inequality we refer to \cite{Stein} and \cite[Theorem 202]{HardyLittlewoodPolya}.

\begin{proposition}\label{prop:convexityoperator}
Let $\beta >1$ and $G_1, G_2 \in L^{\beta}(\Omega)$ be nonnegative functions and $K=(-\Delta)^{-1}$. Then the pointwise inequality holds:
\begin{equation}\label{fneprop}
((KG_1)^{\beta} + (KG_2)^\beta )^{\frac{1}{\beta}} 
\leq K (G_1^\beta + G_2^\beta)^{\frac{1}{\beta}} \quad \text{in} \ \ \Omega.
\end{equation}
Note that the variable $x$ for $KG_i(x)$ and $G_i(x)$ is not indicated to simplify notation.
\end{proposition}

\begin{proof}
This is particular case of Theorem \ref{t:mkw}. Indeed, let $G$ be the Green's function of $-\Delta$, that is, 
$K f = G \ast f$.   Then \eqref{fneprop} follows from Theorem \ref{t:mkw} with $p = \beta$, $d\mu(y) := G(x,y) \, dy$, $\nu = \delta_1 + \delta_2$, $X = \{1,2\}$, $Y = \Omega$ and
$F(i, y) = G_i(y)$.
\end{proof}

\begin{lemma}\label{lemma:monotonicitypower}
Let $m >0$, $m \neq 1$ and $t \in (0,1)$. Then the function
\begin{equation}\label{afc}
q \mapsto  ((1 - t)m^q + t)^{1/q}, \ \ q \in(0, \infty) 
\end{equation}
is increasing.
\end{lemma}
\begin{proof}
Since the derivative of the function defined in \eqref{afc} is 
\begin{equation}
\frac{((1 - t)m^q + t)^{1/q}}{q}  \left(\frac{(1- t)m^q \ln m}{ (1 - t)m^q+ t} - \frac{\ln ((1- t) m^q+ t)}{q} \right) \,,
\end{equation}
it suffices to prove the positivity of the term in parentheses. This is equivalent to 
\[
((1-t)m^q+t) \ln((1-t)m^q+t)<(1-t) m^q \ln m^q +t \ln 1
\]
and the monotonicity follows from the strict convexity of the function $x \in (0,\infty)\mapsto x \ln x$. 
\end{proof}

\begin{proposition}\label{prop:pathsystem}
 Let $f_i \in L^{\frac{p+1}{p}}(\Omega)$, $g_i \in L^{\frac{q+1}{q}}(\Omega)$, $i=1,2$, be positive functions and assume \eqref{eq:hipothesissublinear}. 
If $(f_1, g_1)\neq (f_2, g_2)$, then for each $x\in \Omega$ 
\begin{equation}\label{eq:NewConvSys}
t\mapsto \left((1-t) f_1^{\frac{p+1}{p}} + t f_2^{\frac{p+1}{p}} \right)^{\frac{p}{p+1}} K \left(  \left((1-t) g_1^{\frac{q+1}{q}} + t g_2^{\frac{q+1}{q}} \right)^{\frac{q}{q+1}}\right) \qquad\text{is strictly concave on $[0,1]$.}
\end{equation}
\end{proposition}
\begin{proof}
Arguing as in \eqref{eq:globalargument} of the proof of Lemma \ref{lem:gin}, it is enough to prove the pointwise inequality
\begin{multline}\label{eq:systemineq}
\left((1-t) f_1^{\frac{p+1}{p}} + t f_2^{\frac{p+1}{p}} \right)^{\frac{p}{p+1}} K \left(  \left((1-t) g_1^{\frac{q+1}{q}} + t g_2^{\frac{q+1}{q}} \right)^{\frac{q}{q+1}}\right) \\>  (1 - t) f_1K g_1 + t f_2Kg_2 \,,  \quad \textrm{ for all }  t \in (0,1) , \ \ \textrm{ for all }  x \in \Omega.
\end{multline}

First observe that the condition 
$p\cdot q < 1$  is equivalent to $\frac{p}{p+1} + \frac{q}{q+1} < 1$.
Set $\alpha = \frac{p + 1}{p}$ and let $\beta >0$ be such that
\begin{equation}
\frac{1}{\alpha} + \frac{1}{\beta} = 1.
\end{equation}
Then observe that $\beta < \frac{q+1}{q}$. Applying \eqref{afc} with $m = \frac{a}{b}$, for any $a>0$, $b>0$ with $a\neq b$ and $t \in (0,1)$, 
the function $r \in (0, \infty)\mapsto ((1- t)a^r + t b^r)^{\frac{1}{r}}$
is strictly increasing. Since $K= (- \Delta)^{-1}$ is a strictly monotone operator, it is enough to prove \eqref{eq:systemineq} as non-strict inequality with 
 $\frac{q+1}{q}$ replaced by $\beta$.

After the substitution 
\begin{equation}
F_1 := (1 - t)^{\frac{1}{\alpha}} f_1 , \quad 
G_1 := (1 - t)^{\frac{1}{\beta}} g_1, \quad
F_2 :=t^{\frac{1}{\alpha}} f_2, \quad 
G_2 := t^{\frac{1}{\beta}} g_2,
\end{equation}
we obtain that \eqref{eq:systemineq} (with non-strict inequality)
 is equivalent to 
\begin{equation}\label{eq:systemineqequiv}
(F_1^\alpha + F_2^\alpha)^{\frac{1}{\alpha}} K 
(G_1^\beta + G_2^\beta)^{\frac{1}{\beta}} \geq F_1 K G_1 + F_2 K G_2 
\qquad  t \in (0, 1)\, .
\end{equation}
By H\"older inequality and  Proposition \ref{prop:convexityoperator}
\begin{equation}
F_1 K G_1 + F_2 KG_2
\leq (F_1^\alpha + F_2^\alpha)^{\frac{1}{\alpha}} 
((KG_1)^{\beta} + (KG_2)^\beta )^{\frac{1}{\beta}} \leq 
(F_1^\alpha + F_2^\alpha)^{\frac{1}{\alpha}} K 
(G_1^\beta + G_2^\beta)^{\frac{1}{\beta}}  \,,
\end{equation}
as desired.
\end{proof}

\begin{remark}
If $p >0, \, q> 0$ and $p\cdot q =1$, then \eqref{eq:systemineq} holds true as 
non-strict inequality.
\end{remark}

\begin{remark}
By setting $f_1 = \varepsilon f_2$, $g_1 = \varepsilon g_2$ in \eqref{eq:systemineq} we have (after division by $f_1 K g_1 > 0$)
\begin{equation}
 \left( (1 - t)\varepsilon^{\alpha} + t \right)^{\frac{1}{\alpha}}  \left( (1 - t)\varepsilon^{\beta} + t \right)^{\frac{1}{\beta}}
>  \Big( (1 - t) \varepsilon^{\frac{1}{\alpha} + \frac{1}{\beta}} + t \Big) \,.
\end{equation}
Clearly the opposite (strict) inequality holds true if $\frac{1}{\alpha} + \frac{1}{\beta} > 1$, and $\varepsilon, t > 0$ are sufficiently small. 
Thus $p\cdot q \leq 1$ is an optimal assumption. 
\end{remark}

\begin{proof}[Proof of Theorem \ref{th:hamiltoniansystem}]
Let $A$ be the set of positive solutions of \eqref{eq:more general system} and $\Phi$ be defined by \eqref{phi}. 

\smallbreak 
\noindent \emph{Step 1.} The set $A$ contains at most one element.

 It is enough to show that conditions a)-c) of Theorem \ref{th:abstracttheorem} are satisfied. Given $(f_i, g_i)$ with $f_i, g_i >0$, for $i=1,2$, and such that $(f_1, g_1) \neq (f_2, g_2)$, consider the path
\begin{equation}\label{toworkrpoof}
\gamma(t) = \left( \left((1-t) f_1^{\frac{p+1}{p}} + t f_2^{\frac{p+1}{p}} \right)^{\frac{p}{p+1}},  \left((1-t) g_1^{\frac{q+1}{q}} + t g_2^{\frac{q+1}{q}} \right)^{\frac{q}{q+1}}\right),  \ \ t \in[0,1],
\end{equation}
which connects $(f_1, g_1)$ to $(f_2, g_2)$, i.e. condition a) from Theorem \ref{th:abstracttheorem}. In addition, from \eqref{eq:NewConvSys} we infer that
\begin{equation}\label{eq:pathsystemineqagain}
t \mapsto \Phi(\gamma(t)) \quad \text{is strictly convex on $[0,1]$}
\end{equation}
and condition c) of Theorem \ref{th:abstracttheorem} follows.
On the other hand, any positive solution $(f, g)$ of \eqref{euler-lagrange} 
corresponds to a positive solution
$(u,v) :=( \phi_ p^{-1}(f), \phi_ q^{-1}(g)) \in C^2(\overline{\Omega})\times C^2(\overline{\Omega})$ of \eqref{eq:more general system} and, by the Hopf Lemma, 
we infer that $\frac{\partial u}{\partial \nu} < 0$ and $\frac{\partial v}{\partial \nu} < 0$ on $\partial \Omega$. 
Then by Lemma \ref{lemma:gradientboundarycomparison} we have $\delta^{-1} u_1 \leq u_2 \leq \delta u_1$ and 
 $\delta^{-1} v_1 \leq v_2 \leq \delta v_1$ for any  positive solutions $(u_1, v_1)$ and $(u_2, v_2)$ of \eqref{eq:more general system}.
 Consequently $\delta^{-1} f_1 \leq f_2 \leq \delta f_1$ and $\delta^{-1} g_1 \leq g_2 \leq \delta g_1$ and Corollary \ref{corollary_of_lemma:gpg} yields condition b) of Theorem \ref{th:abstracttheorem}. Therefore, $A$ is empty or a singleton.

\smallbreak

\noindent \emph{Step 2.} Existence of a positive solution (global minimizer of $\Phi$). 

Let $(f,g) \in L^{\frac{p+1}{p}}(\Omega)\times L^{\frac{q+1}{q}}(\Omega)$ and let $\alpha, \beta$ be as in the proof of Proposition 
\ref{prop:pathsystem}. Then, since $\|Kg\|_{L^s} \leq C \|g\|_{L^s}$ for any $s > 0$ 
we infer from the H\"older and Young inequalities (note that $\beta < \frac{q + 1}{q}$) that for any $\epsilon > 0$
\begin{align*}
\int_{\Omega} f K g \, dx &\leq 
C \|f\|_{L^{\alpha}}  \|g\|_{L^\beta} \leq \epsilon (\|f\|_{L^{\frac{p+1}{p}}}^{\frac{p+1}{p}}  + 
\|g\|_{L^{\frac{q+1}{q}}}^{\frac{q+1}{q}}) + C(\epsilon) \,,
\end{align*}
and consequently $\Phi$ is coercive, and in particular bounded from below. Since $Kg\in W^{2,\frac{q+1}{q}}(\Omega)\cap W^{1,\frac{q+1}{q}}_0(\Omega)$ and we have the compact embedding $W^{2, \frac{q+1}{q}}(\Omega) \subset \subset L^{\frac{q+1}{q}}(\Omega) \subset L^{p+1}(\Omega)$, we infer that $\Phi$ is lower semi-continuous and thus attains its 
global minimum, which is negative (indeed, taking $\varphi$ a $C^\infty(\overline \Omega)$ positive function, $\Phi(\varepsilon \varphi,(\varepsilon \varphi)^\frac{q(p+1)}{p(q+1)})<0$ for small $\varepsilon>0$) at a point $(f_0,g_0)$ with $f_0,g_0>0$ in $\Omega$. The latter statement follows from the maximum principle. 
\end{proof}

\section{Nonlinear eigenvalue problems from Mathematical Physics}\label{sec:RN}

In this section we apply our main results to elliptic equations and systems appearing in Mathematical Physics. More precisely, we approach eigenvalue type-problems, that is, problems which can be written in the form
\[
L_1u_1=\omega_1 u_1,\quad \ldots, \quad L_k u_k=\omega_k u_k \qquad (k\geq 1),
\]
where $L_i$ are nonlinear differential operators. There are two different kind of natural questions: firstly, given $\omega_1,\ldots, \omega_k$, does there exist 
a unique positive solution to the problem? Secondly, does there exist a unique positive eigenvalue satisfying the constraint $\int u_i^2=1$? These questions are, in general, not equivalent: in the second framework, the eigenvalues are not fixed a priori, and appear as Lagrange multipliers. We will deal with these issues in the case of a quasilinear Schr\"odinger equation (Section \ref{subsec:quasilinear}) and for a Gross-Pitaevskii--type system (Section \ref{subsecGP}), both for $\Omega$ bounded and $\Omega={\mathbb R}^N$.

\subsection{Defocusing Schr\"odinger equation}\label{subsec:quasilinear}

Consider the equation
\begin{equation}
  \label{eq:gs1}
  -\Delta u-u\Delta u^2+ V(x) u + u^3 =\omega u\qquad  \text{in $\Omega$},
\end{equation}
where $V(x)$ is the trapping potential $|x|^2$ in case $\Omega={\mathbb R}^N$, or $V\in L^\infty(\Omega)$ and $u=0$ on the boundary if $\Omega$ is a bounded regular domain, and $N\geq 1$. This equation arises when looking for standing waves $\phi(t,x)=e^{\imath \omega t}u(x)$ of the
quasilinear defocusing Schr\"odinger equation 
\begin{equation}
   \label{eq:schr1}
 i \partial_t\phi-\Delta\phi-\phi\Delta |\phi|^2 + V \phi +  \phi^3=0\quad  \text{in $(0,\infty)\times \Omega$},
 \end{equation}
which serves as a model in many physical situations, for which we refer to \cite{ColinJeanjeanSquassina, PoppSchmittWang}.

The first results for \eqref{eq:gs1} appeared in \cite{LiuWangI, PoppSchmittWang}, where for $\Omega={\mathbb R}^N$ and $\omega=0$, a constrained minimization is used
to prove existence of solutions. 
Difficulties arise due to the presence of the term $u\Delta u^2$, and are related to the existence of three different scales in the equation. In the variational 
setting the term $\int u^2 |\nabla u|^2\, dx$ is not well defined in $H^1({\mathbb R}^N)$, while the natural set $\{u\in H^1({\mathbb R}^N); \int u^2 |\nabla u|^2\, dx<\infty\ \}$ is not a vector space. To treat $u\Delta u^2$ a substitution $u=f(v)$ (cf. \eqref{odef}) was introduced in \cite{ColinJeanjean, LiuWangWangII}, 
which reduced \eqref{eq:gs1} to a semilinear equation, with a more complicated nonlinear zero order term. This reduction, that has also been used for example in \cite{AdachiWatanabe, ColinJeanjeanSquassina}, allows one to work in the $H^1$ setting. 

We will use the dual method as in \cite{ColinJeanjean}. Let $f$ be the odd function such that
\begin{equation}\label{odef}
f'(t)=\frac{1}{\sqrt{1+2f^2(t)}} \quad \text{in } (0,\infty), \quad f(0)=0.
\end{equation}
If $u$ is a solution to \eqref{eq:gs1}, by \cite[\S 2]{ColinJeanjean}, then $v=f^{-1}(u)$ is a solution of
\begin{equation}\label{eq:dual_formulation_quasilinear}
-\Delta v + (V(x)-\omega)\frac{f(v)}{\sqrt{1+2f^2(v)}}+ \frac{(f(v))^3}{\sqrt{1+2f^2(v)}}=0.
\end{equation}
We recall from \cite[Lemma 2.1]{ColinJeanjean} that
\begin{equation}\label{eq:study_of_f}
\frac{f(t)}{t}\to 1 \quad \text{as } t\to 0, \qquad \frac{f(t)}{\sqrt{t}} \to 2^{1/4} \quad \text{ as } t\to +\infty,
\end{equation}
and from \cite[Lemma 2.2]{ColinJeanjean} that
\begin{equation}\label{eq:lemma2.2CJ}
\frac{1}{2} f(t) \leq \frac{t}{\sqrt{1 + 2 f^2(t)}} \leq f(t) \qquad 
\textrm{ for all } t \in {\mathbb R}.
\end{equation}
For every fixed $\omega \in {\mathbb R}$, the solutions of \eqref{eq:dual_formulation_quasilinear} are critical points of the $C^1$-action functional
$\mathcal{A}_\omega : H \to \mathbb{R}$ with
\[
\mathcal{A}_\omega(v) =\frac{1}{2}\int_\Omega |\nabla v|^2\, dx+\frac{1}{2}\int_\Omega (V(x)-\omega) (f(v))^2\, dx +\frac{1}{4}\int_\Omega (f(v))^4\, dx,
\]
where $H = \{ u \in H^1({\mathbb R}^N);  \int_{{\mathbb R}^N}  |u|^2 |x|^2\, dx < \infty\}$ with the norm 
\[
\| u \|^2_H = \int_{\mathbb{R}^N} |\nabla u|^2  + |u|^2 |x|^2\, dx 
\]
 if $\Omega = {\mathbb R}^N$,
and $H = H^1_0(\Omega)$ with the standard norm if $\Omega$ is a bounded regular domain. Observe that $\mathcal{A}_\omega\in C^1(H,\mathbb{R})$ since \eqref{eq:lemma2.2CJ}  implies that $f^2(t)+2f^4(t)\leq 4t^2$.

On the other hand, we can consider the constraint
\[
\mathcal{M} := \left\{v \in H; \ \int_\Omega (f(v))^2\, dx=1\right\}
\]
and the energy $\mathcal{E}:\mathcal{M}\to {\mathbb R}$,
\[
\mathcal{E}(v) := \frac{1}{2}\int_\Omega |\nabla v|^2\, dx+ \frac{1}{2}\int_{\Omega} V(x) (f(v))^2\, dx+ \frac{1}{4}\int_{\Omega} (f(v))^{4}\, dx.
\]
 In this second framework, one looks for critical points of $\mathcal{E}|_{\mathcal{M}}$ and $\omega$ appears as a Lagrange multiplier.

When $\Omega$ is bounded we prove uniqueness of positive critical points both for $\mathcal{A}_\omega$ and $\mathcal{E}|_{\mathcal{M}}$, while for $\Omega={\mathbb R}^N$, we deduce uniqueness of positive critical points of $\mathcal{A}_\omega$. The following statements are, 
to the best of our knowledge, new.

\begin{theorem}\label{thmQLNLS_bounded}
Let $\Omega$ be a bounded regular domain of $ {\mathbb R}^N$, with $N \geq 1$, and assume that $V\in L^\infty(\Omega)$.  Then,
\begin{enumerate}[i)]
\item  for each $\omega\in \mathbb{R}$ fixed, there exists at most one positive solution to \eqref{eq:gs1} with $u=0$ on $\partial \Omega$, that is, $\mathcal{A}_\omega$ has at most one positive critical point, which exists if, and only if, $\omega > \lambda_V$. Here $\lambda_V$ stands for the first eigenvalue of $(-\Delta + V(x)I, H^1_0(\Omega))$.
\item there exists exactly one positive critical point of $\mathcal{E}|_{\mathcal M}$.
\end{enumerate}
\end{theorem}

\begin{theorem}\label{thmQLNLS_unbounded}
Let $\Omega={\mathbb R}^N$, with $N\geq 1$, and take $V(x)=|x|^2$. 
For each $\omega\in \mathbb{R}$ fixed, there exists at most one positive solution of \eqref{eq:gs1}, that is, at most one positive critical point of $\mathcal{A}_\omega$, which exists if, and only if, $\omega > \lambda_V$. Here $\lambda_V$ stands for the first eigenvalue of $(-\Delta + V(x)I, H)$.
\end{theorem}

In order to prove Theorem \ref{thmQLNLS_unbounded}, we will need to deduce the sharp decay at infinity for all positive solutions of \eqref{eq:dual_formulation_quasilinear}. Since $v=f^{-1}(u)$ and $f(t)\sim t$ as $t\to 0$, this also gives the sharp decay result for positive solutions of our original problem \eqref{eq:gs1}. Since we believe this to be of independent interest, we state these results as a theorem.
\begin{theorem}\label{prop:sharpdecay_qlS}
 Let $v\in H$ be any positive solution of problem
\begin{equation}\label{eq:aeq}
-\Delta v + (|x|^2-\omega)\frac{f(v)}{\sqrt{1+2f^2(v)}}+ \frac{(f(v))^3}{\sqrt{1+2f^2(v)}}=0 \qquad \text{ in } {\mathbb R}^N.
\end{equation}
Then
\begin{equation}\label{eq:shd}
v(x)\sim |x|^\frac{\omega-N}{2}e^{-\frac{|x|^2}{2}} \qquad \text{ as } |x|\to \infty,
\end{equation}
that is, there exist $C_1,C_2>0$ such that
\[
C_1 |x|^\frac{\omega-N}{2}e^{-\frac{|x|^2}{2}} \leq v(x) \leq C_2 |x|^\frac{\omega-N}{2}e^{-\frac{|x|^2}{2}} \qquad \text{ for large } |x|.
\]
Consequently, if $u$ is a positive solution of 
\[
-\Delta u-u\Delta u^2+ |x|^2 u + u^3 =\omega u\qquad  \text{in ${\mathbb R}^N$},
\]
then
\[
u(x)\sim |x|^\frac{\omega-N}{2}e^{-\frac{|x|^2}{2}} \qquad \text{ as } |x|\to \infty.
\]
\end{theorem}

\smallbreak

Let us proceed to the proof of the main results. Next, we will build suitable paths to set the proofs of Theorems \ref{thmQLNLS_bounded} and \ref{thmQLNLS_unbounded} in the framework of Theorem \ref{th:abstracttheorem} and Corollary \ref{cor:abstractcor}. The first step is to verify some convexity of the gradient of these new paths.

\begin{lemma}\label{lemma:inequalitygradientNEW}
Let $u, v \in H \cap W^{1,\infty}(\Omega)$ such that $u,v>0$ in $\Omega$. Set
\begin{equation}\label{eq:boundforgammaNEW}
\gamma(t) = f^{-1}\left( \sqrt{(1-t) f^2(u) + t f^2(v)}\right), \quad \textrm{ for all } t \in [0,1].
\end{equation}
Then $\gamma: [0,1] \to H$ is well defined and, pointwise in $\Omega$, 
\begin{equation}\label{eq:fgradientconvex}
t \mapsto |\nabla \gamma (t)|^2 \quad \text{is convex on $[0,1]$}.
\end{equation}
\end{lemma}

\begin{proof} First observe that $\min\{ u, v\} \leq \gamma(t) \leq \max\{u,v\}$ for all $t \in [0,1]$ since $f^2$ is increasing. Then, once \eqref{eq:fgradientconvex} is proved, we infer that $\gamma(t) \in H$ for all $t \in [0,1]$.

We use Lemma 
\ref{lem:gin} and so we show that its hypotheses are satisfied. In this case we have $Q(z)=f^2(z)$, $M(z) = z^2$ and we prove that $(z_1, z_2) \mapsto F(z_1, z_2) = F_1(z_1)F_2(z_2)$ is concave on $(0, \infty) \times [0, \infty)$ where $F_1(z_1) = Q'\circ Q^{-1}(z_1) = \frac{2 z_1^{1/2}}{(1 + 2z_1)^{1/2}}$ by \eqref{odef} and $F_2(z_ 2) = z_2^{1/2}$. It is simple to verify that $F_1$ and $F_2$ are concave and, by direct computations, $\left( \frac{F_ 2}{F_2'} \right)'- 1 =1$ and  $\left( \frac{F_ 1(z_1)}{F_1'(z_1)} \right)'- 1 =1 + 8 z_1 \geq 1$. Hence, as observed at Remark \ref{remark:concavity}, $F$ is concave.
\end{proof} 
Next, we use Lemma \ref{lemma:gpg} to prove that $\gamma$ defined by \eqref{eq:boundforgammaNEW} is locally Lipschitz at $t=0$ whenever $u$ and $v$ are comparable.

\begin{lemma}\label{lemma:goodpathNEW}
Consider $u, v \in H$ such that
\begin{equation}\label{eq:hipderivativeNEW}
u, v > 0 \ \ \text{in} \ \ \Omega \ \ \text{and there is }  \delta \geq 1 \textrm{ such that } \ \ \delta^{-1} v \leq u \leq \delta v \ \ \text{in} \ \ \Omega
\end{equation}
and let $\gamma$ be as in \eqref{eq:boundforgammaNEW}. Then,  $\gamma: [0,1] \to H$ satisfies $\gamma(0) = u$, $\gamma(1)=v$, and 
$\gamma$ is locally Lipschitz at $t = 0$.
\end{lemma}

\begin{proof} 
It is obvious that $\gamma(0) = u$, $\gamma(1)=v$ and so we just need to verify conditions a), b), and c) of Lemma \ref{lemma:gpg}. Take also into account that we can argue as in \eqref{eq:fpr} to infer that $\gamma$ is locally Lipschitz at $t=0$ with respect to the term $w \mapsto \int_{\mathbb{R}^N} |x|^2|w|^2 \,dx $ in the case of $\Omega= \mathbb{R}^N$.

In this case we have $Q = f^2$ and so, by \eqref{odef}, $Q' (z) = \frac{2f(z)}{\sqrt{1 + 2 f^2(z)}}$, which implies condition a), and $Q''(z) = \frac{1}{(1 + 2f^2(z))^2} $. Then b) and c) follows from  \eqref{eq:study_of_f} and \eqref{eq:lemma2.2CJ}.
\end{proof}

Now, once the paths are built and their convexity properties established, we prove Theorem \ref{thmQLNLS_bounded}.

\begin{proof}[Proof of Theorem {\rm{\ref{thmQLNLS_bounded}}}]
Here we take either
\[
A_{\omega}=\{u\in H; \, u>0 \text{ and } \mathcal{A}_\omega'(u)=0\} \text{ or } A_{\mathcal E}=\{u\in \mathcal{M}; \, u>0 \text{ and } \mathcal{E}'|_{\mathcal M}(u)=0\}.
\]

\noindent \emph{Step 1.} Sets $A_{\omega}$ and $A_{\mathcal E}$ are empty or singletons.

Let $u, v\in A_{\omega}$, or $u,v\in A_{\mathcal E}$ (to shorten the proof, we will prove both statements of the theorem at the same time).
By combining Lemma \ref{lemma:inequalitygradientNEW} with the strict convexity of the map $t \mapsto t^2$,  we immediately have for 
$\gamma$ defined in \eqref{eq:boundforgammaNEW}, that 
for $u \neq v$
\[
t \mapsto  \mathcal{A}_\omega(\gamma(t)) \quad \text{and} \quad t\mapsto  \mathcal{E}(\gamma(t)) \quad \text{are strictly convex on [0,1]}
\]
and $\gamma(t)\in \mathcal{M}$ for every $t \in [0,1]$ provided $u, v \in \mathcal{M}$. Thus, in order to apply Theorem \ref{th:abstracttheorem} for $A_{\omega}$ and Corollary \ref{cor:abstractcor} for $A_\mathcal{E}$, one needs 
to prove that $\gamma$ is locally Lipschitz at $0$. For that, we will prove a comparison result like \eqref{eq:hipderivativeNEW}. Let $\bar u$ and $\bar v$ solve \eqref{eq:dual_formulation_quasilinear}, which (for positive solution) can be written in the form
\[
-\Delta w+a(w)w=0, \quad \text{ with } a(w)=(V(x)-\omega)\frac{f(w)}{w\sqrt{1+2f^2(w)}}+ \frac{f^3(w)}{w\sqrt{1+f^2(w)}}\in L^\infty.
\]
Then Hopf's Lemma implies that $1/\delta \leq \bar v/\bar u\leq \delta$ on $\overline \Omega$ as we desired and so we can apply Lemma \ref{lemma:goodpathNEW}.
\smallbreak

\noindent \emph{Step 2.} 
If $\omega \leq \lambda_V$, then $A_{\omega}$ is empty.  This follows after testing \eqref{eq:gs1} by $u$.

\smallbreak

\noindent \emph{Step 3.} The set $A_{\mathcal E}$ is not empty. Moreover, if $\omega > \lambda_V$, then $A_{\omega}$ is not empty as well.

Observe that both $A_{\omega}$ and $A_{\mathcal E}$ contain a global minimum of the corresponding functionals.
In fact, both functionals are bounded from below and coercive: $\mathcal{E}(v)\geq \frac{1}{2} \int_{\Omega} |\nabla v|^2 \, dx + C_1 |\Omega|$, while $\mathcal{A}_{\omega}(v) \geq \frac{1}{2} \int_{\Omega} |\nabla v|^2 \, dx+ C_ 2|\Omega|$, where $C_1$ and $C_2$ are respectively  the absolute minimum levels of the polynomials $t \mapsto -\frac{1}{2}\|V\|_{\infty} t^2 + \frac{1}{4} t^4$ and $t\mapsto-\frac{1}{2}( \|V\|_{\infty} +\omega) t^2 + \frac{1}{4}t^4$, and we have compact embeddings of $H^1_0(\Omega)$ into $L^2(\Omega)$.  By combining $\omega> \lambda_V$ with the behaviour of $f$ close to $t=0$, we can prove that $\varepsilon\mapsto \mathcal{A}_\omega(\varepsilon \varphi_1)$ is negative for $\varepsilon\sim 0$, where $\varphi_1$ is a positive eigenfunction of  $(-\Delta + V(x)I, H^1_0(\Omega))$, and therefore $\min_{H} \mathcal{A_{\omega}}$ is negative. 
\end{proof}

In order to prove Theorem \ref{thmQLNLS_unbounded}, we need first to show the sharp decay of the positive solutions of our problem as in Theorem \ref{prop:sharpdecay_qlS}. This result is a consequence of the following very general result.

\begin{lemma}[{\cite[Proposition 6.1]{MorozVanSchaftingen}}]\label{lemma:ExponentialDecay}
Let $N\geq 1$, $\gamma<2$, and fix $\rho\geq 0$ and a nonnegative function $W\in C^1([\rho,\infty))$. If
\begin{equation}\label{eq:assumptions_ExponentialDecay}
\lim_{s \to +\infty} W(s)>0 \quad \text{ and } \quad \lim_{s\to +\infty} W'(s) s^{1+\beta}=0\ \text{for some $\beta>0$}, 
\end{equation}
then there exists a nonnegative radial function $h:{\mathbb R}^N\setminus \overline{B_\rho(0)}\to {\mathbb R}$ such that
\[
-\Delta h(x)+\frac{W^2(|x|)}{|x|^\gamma} h(x)=0 \qquad \textrm{ for all }  x\in {\mathbb R}^N\setminus \overline{B_\rho(0)}
\]
and
\[
\lim_{|x|\to \infty} h(|x|) |x|^{\frac{N-1}{2}-\frac{\gamma}{4}}\exp\left(\displaystyle \int_\rho^{|x|} \frac{W(s)}{s^\frac{\gamma}{2}}\, ds\right)=1
\]
\end{lemma}

\begin{proof}[Proof of Theorem {\rm\ref{prop:sharpdecay_qlS}}]
\emph{Step 1.} We claim that $v(x)\to 0$ as $|x|\to \infty$.

Observe that
\[
-\Delta v+(|x|^2+1)v \frac{f(v)}{v\sqrt{1+2f^2(v)}}+ v \frac{f^3(v)}{v\sqrt{1+2f^2(v)}}=(1+\omega) v \frac{f(v)}{v\sqrt{1+2f^2(v)}}\,,
\]
and therefore, $v$ solves
\[
-\Delta v +  a(v)v \leq (1+\omega) b(v) v,
\]
with 
\[
a(t)=\frac{f(t)}{t\sqrt{1+2f^2(t)}}+\frac{f^3(t)}{t\sqrt{1+2f^2(t)}},\qquad b(t)=\frac{f(t)}{t\sqrt{1+2f^2(t)}}.
\]
By \eqref{eq:study_of_f} and \eqref{eq:lemma2.2CJ}, we have
\[
a,b\in L^\infty(0,\infty), \ \text{ and } a \text{ is bounded away from 0}.
\]
In particular, there exist $\kappa_1,\kappa_2>0$ such that
\[
-\Delta v+\kappa_1 v\leq \kappa_2 v \text{ in } {\mathbb R}^N.
\]
Now a classical Brezis-Kato/Moser Iteration Scheme yields that $v\in L^\infty({\mathbb R}^N)$ (see \cite[pages 1264-1265]{NTTV_JEMS}). Moreover, reasoning as in \cite[\S 3.4]{GilbargTrudinger} we have $|\nabla v|\in L^\infty({\mathbb R}^N)$ which, combined with $v \in H$ and  the mean value theorem, implies the claim of this step.

\smallbreak
\noindent \emph{Step 2.} 
As an intermediate step, we prove that there exist $C > 0$ and $R > 0$
such that  
\begin{equation}\label{eq:exponential_nonopt}
v(x) \leq C e^{-\frac{|x|^2}{4}} \qquad 
\textrm{ for all } |x| \geq R \,.
\end{equation}
Rewrite \eqref{eq:aeq} as 
\begin{equation}\label{eq:ede}
-\Delta v+(|x|^2-\omega)v=(|x|^2-\omega)\left(v-\frac{f(v)}{\sqrt{1+2f^2(v)}}\right)-\frac{f^3(v)}{\sqrt{1+2f^2(v)}}.
\end{equation}
Since $v\to 0$ as $|x|\to \infty$, then by 
\eqref{eq:study_of_f}, 
for each $\varepsilon>0$ there exists $R>0$ such that the right-hand side of \eqref{eq:ede} is less than or equal to $(|x|^2-\omega)\varepsilon v$ outside $B_R(0)$, and consequently
\[
-\Delta v +(|x|^2-\omega)(1-\varepsilon) v\leq 0 \quad \text{ in } {\mathbb R}^N\setminus B_R(0).
\]
Set $\rho:=\max\{R,2\sqrt{\omega}\}$. Then $W(s):=\sqrt{(1-s^{-2}\omega)(1-\varepsilon)}$
satisfies \eqref{eq:assumptions_ExponentialDecay}, since $W'(s)\sim Cs^{-3}$ as $s\to \infty$. Then we apply Lemma \ref{lemma:ExponentialDecay} with $\gamma=-2$, obtaining the existence of a positive function $h$ satisfying
\[
-\Delta h(x)+(|x|^2-\omega)(1-\varepsilon) v=0 \text{ in } {\mathbb R}^N\setminus \overline{B_\rho(0)}
\]
and
\begin{equation}\label{eq:H-first-estimate}
h(x)\sim |x|^{-\frac{N}{2}}\exp\left( -\int_\rho^{|x|} \sqrt{(s^2-\omega)(1-\varepsilon)}\,ds \right) \qquad \text{ as } |x|\to \infty.
\end{equation}
Thus, for $|x|$ sufficiently large, we infer that 
\[
h(x)\leq C e^{-\frac{|x|^2}{4}} \qquad |x| \geq R.
\]
By replacing $h$ by $a h(x)$, where $a$ is a sufficiently large constant, we can assume that $v(x)\leq h(x)$ on $\partial B_R(0)$. Hence, by the maximum principle, $v(x) \leq h(x)$ on ${\mathbb R}^N\setminus \overline{B_R(0)}$ and thus the claim \eqref{eq:exponential_nonopt} follows.

\smallbreak
\noindent \emph{Step 3.} 
Finally, we show that the sharp decay \eqref{eq:shd}
holds.

By \eqref{eq:study_of_f}, for small $v$ the expressions 
\[
\left| 1-\frac{f(v)}{v\sqrt{1+2f^2(v)}} \right| \textrm{ and } \frac{f^3(v)}{v\sqrt{1+2f^2(v)}}
\]
are bounded by $c v^2 \leq C e^{-|x|^2/2}$. Hence we have, for some $C>0$, that
\[
-\Delta v+(|x|^2-\omega-C e^{-\frac{|x|^2}{4}})\leq 0,\qquad -\Delta v+(|x|^2-\omega+C e^{-\frac{|x|^2}{4}})\geq 0
\]
in ${\mathbb R}^N\setminus \overline{B_R(0)}$. Observing that $W_{\pm}(s):=\sqrt{1-s^{-2}\omega \pm C s^{-2}e^{-\frac{s^2}{4}}}$ satisfy the assumptions of Lemma \ref{lemma:ExponentialDecay},  with $W'(s)\sim s^{-3}$ as $s\to \infty$, for a $\tilde\rho \geq R$ sufficiently large,  there exist positive functions $h_+,h_-$ with:
\[
-\Delta h_{\pm}+ (|x|^2-\omega \pm C e^{-\frac{|x|^2}{4}})h_\pm =0 \qquad \textrm{ for all }  |x|\geq \tilde\rho,
\]
and
\[
h_\pm(x)\sim  |x|^{-\frac{N}{2}}\exp \left(-\int_\rho^{|x|} \sqrt{s^2-\omega \pm C e^{-s^2/4}}   \, ds\right) \quad \text{ as } |x|\to \infty.
\]
Thus, by the maximum principle, there exist positive constants $a_+$ and $a_-$ such that
\[
a_+ h_+ \leq u_i \leq a_-h_- \qquad \textrm{ for all }  |x|\geq \tilde\rho.
\]
Since
\[
\sqrt{s^2-\omega - C e^{-s^2/4}}\geq \sqrt{s^2-\omega}-\sqrt{C e^{-s^2/4}}\quad  \text{ and }\quad
\sqrt{s^2-\omega + C e^{-s^2/4}}\leq \sqrt{s^2-\omega}+\sqrt{C e^{-s^2/4}} ,
\]
and $e^{-s^2/2}$ is an integrable function, we get the existence of two constants $C_1,C_2>0$ such that
\[
C_1 |x|^{-\frac{N}{2}}\exp\left(- \int_{\tilde\rho}^{|x|} \sqrt{s^2-\omega} \, ds \right) \leq u_i(x)  \leq C_2 |x|^{-\frac{N}{2}}\exp\left(- \int_{\tilde\rho}^{|x|} \sqrt{s^2-\omega} \, ds \right). 
\]
and the statement now follows from 
\begin{equation}
 \int_{\tilde\rho}^{|x|}\sqrt{s^2-\omega}\, ds= \frac{1}{2}\left(|x| \sqrt{|x|^2-\omega}-\omega \ln(\sqrt{|x|^2-\omega}+|x|)  \right) + C(\tilde{\rho}) 
 \end{equation}
 and 
\begin{equation}
|x|^2 \geq |x|\sqrt{|x|^2-\omega} \geq |x|^2 - M \,, \qquad 
|x| \geq  \sqrt{|x|^2-\omega} \geq m |x|
\end{equation}
for appropriate $M, m > 0$ depending on $\omega$ and all sufficiently large $|x|$. 
\end{proof}

\begin{proof}[Proof of Theorem {\rm\ref{thmQLNLS_unbounded}}]

Take  
\[
A_\omega=\{v\in H; \ v>0 \text{ and } \mathcal{A}_\omega'(v)=0 \}.
\]
\noindent \emph{Step 1.} We claim that $A_{\omega}$ is empty or a singleton.

Suppose $u, v \in A_{\omega}$ with $u \neq v$. As in the bounded domain case we have that
\[
t \mapsto  \mathcal{A}_\omega(\gamma(t))\quad  \text{ is strictly convex on [0,1]}
\]
with $\gamma$ defined in \eqref{eq:boundforgammaNEW}.  
It remains to show that, given $u, v \in A_{\omega}$,
there are $C_1, C_2 > 0$ such that  
\[
C_1\leq \frac{v(x)}{u(x)}\leq C_2 \qquad \text{ for all }  x\in {\mathbb R}^N,
\]
which follows from Theorem \ref{prop:sharpdecay_qlS}.

\smallbreak \noindent \emph{Step 2.} We show that  $A_{\omega}$ is not empty if, and only if, $\omega> \lambda_V$. 

Since $V(x)=|x|^2$ implies a compact embedding $H \hookrightarrow L^2({\mathbb R}^N)$, we infer that
\[
\lambda_V = \inf_{v \in H} \frac{\int_{\mathbb{R}^N} |\nabla v|^2  + V(x) v^2 dx}{\int_{\mathbb{R}^N} v^2 dx},
\]
the first eigenvalue of $(-\Delta + V(x)I, H)$, is positive.

By testing \eqref{eq:gs1} by $u$ we see that $A_{\omega}$ is empty if $\omega \leq \lambda_V$.

Now, if $\omega > \lambda_V$, as in the bounded domain case, we have that $m:=\inf_{H} \mathcal{A_{\omega}}$ is negative. 
 Moreover, observe that
\begin{equation}\label{eq:goodgradientRN}
\mathcal{A_{\omega}} (v) \geq \frac{1}{2} \int_{\mathbb{R}^N} |\nabla v|^2 \, dx + C_3 |B(0, \sqrt{|\omega|})| 
\end{equation}
where $C_3$ is the minimum of the function $t \mapsto - \frac{1}{2}|\omega| t^2 + \frac{1}{4} t^4$ for $t \in [0, \infty)$. 
Next, we prove that there exists a positive function $v \in H$ that satisfies  $\mathcal{A_{\omega}}(v) = m := \inf_{H} \mathcal{A_{\omega}}$, hence $v\in A_\omega$. Observe that, unlike for $\Omega$ bounded, it is not immediate  that minimizing sequences are bounded in $\|\cdot\|_H$, due to the term $\int_{\mathbb{R}^N}|x|^2 (v(x))^2\, dx$ in the definition of the norm.

However, since $\mathcal{A_{\omega}}$ is bounded from below, there is 
a minimizing sequence $(v_n) \subset H$ of $\mathcal{A_{\omega}}$. 
Since $\mathcal{A}_\omega(|v|) = \mathcal{A}_\omega(v)$, we can suppose that $v_n \not \equiv 0$, $v_n \geq 0$ and $\mathcal{A}_\omega(v_n) <0$ for all $n\in \mathbb{N}$. In addition, since 
$\mathcal{A}_\omega(v^*) \leq \mathcal{A}_\omega(v)$ for all $v\geq 0$ in $H^1(\mathbb{R}^N)$, where $v^*$ stands for the Schwarz symmetrization of $v$, we can assume that $v_n$ are Schwarz symmetric.

From
\[
\mathcal{A}_\omega(v_n) =\frac{1}{2}\int_{\mathbb{R}^N} |\nabla v_n|^2\, dx+\frac{1}{2}\int_{\mathbb{R}^N} (|x|^2-\omega) (f(v_n))^2\, dx +\frac{1}{4}\int_{\mathbb{R}^N} (f(v_n))^4\, dx  <0,
\]
we infer that
\begin{multline*}
\frac{1}{2}\int_{\mathbb{R}^N} |\nabla v_n|^2\, dx+\frac{1}{2}\int_{\mathbb{R}^N} ||x|^2-\omega| (f(v_n))^2\, dx +\frac{1}{4}\int_{\mathbb{R}^N} (f(v_n))^4\, dx \leq \int_{|x|^2 < \omega} (\omega -|x|^2) (f(v_n))^2\, dx\\
\leq \frac{1}{8}\int_{|x|^2 < \omega} (f(v_n))^4\, dx + C_0(\omega) \leq \frac{1}{8}\int_{\mathbb{R}^N} (f(v_n))^4\, dx + C_0(\omega)
\end{multline*}
for some $C_0(\omega) > 0$, and hence
\begin{equation}\label{eq:boundmany1}
\frac{1}{2}\int_{\mathbb{R}^N} |\nabla v_n|^2\, dx+\frac{1}{2}\int_{\mathbb{R}^N} ||x|^2-\omega| (f(v_n))^2\, dx +\frac{1}{8}\int_{\mathbb{R}^N} (f(v_n))^4\, dx 
\leq C_0(\omega).
\end{equation}

From \eqref{eq:boundmany1} we infer that
\begin{equation}\label{eq:boundf2}
\begin{aligned}
\int_{\mathbb{R}^N} (f(v_n))^2\, dx &\leq \int_{|x|^2 < 2 \omega} (f(v_n))^2\, dx + \frac{1}{\omega}\int_{|x|^2 \geq 2 \omega}(|x|^2 - \omega)(f(v_n))^2 \, dx \leq C_1(\omega) + \int_{\mathbb{R}^N} (f(v_n))^4 \, dx \\
&\leq C_2(\omega)\,.
\end{aligned}
\end{equation}
Hence, from \eqref{eq:boundmany1}, \eqref{eq:boundf2} 
one has 
\begin{equation}\label{eq:boundmany}
\frac{1}{2}\int_{\mathbb{R}^N} |\nabla v_n|^2\, dx+\frac{1}{2}\int_{\mathbb{R}^N} (|x|^2 + 1) (f(v_n))^2\, dx +\frac{1}{8}\int_{\mathbb{R}^N} (f(v_n))^4\, dx 
\leq C_3(\omega).
\end{equation}
Since $v_n \in H^{1}(\mathbb{R}^N)$ is Schwarz symmetric, 
then $v_n \in C (\mathbb{R}^N\backslash\{0\})$ and so $v_n(1) = M_n$ is well defined. As $v_n$ is decreasing in the radial direction, 
\eqref{eq:boundmany} implies $M_n \leq C_4(\omega)$. Indeed,
\[
(f(M_n))^4 |B_1(0)| \leq \int_{|x| \leq 1} (f(v_n))^4 \, dx\leq 8 \,C_3(\omega)\,.
\]
From \eqref{eq:study_of_f} there exists a positive constant $C_5(\omega)$ such that $t \leq C_5(\omega) f(t)$ for all $t \in [0, C_4(\omega)]$. Hence, from \eqref{eq:lemma2.2CJ} and \eqref{eq:boundmany},
\begin{multline*}
\int_{\mathbb{R}^N} |x|^2 v_n^2 \, dx = \int_{|x| \leq 1} |x|^2 v_n^2 \, dx  + \int_{|x| > 1} |x|^2 v_n^2 \, dx \leq  \int_{|x| \leq 1} v_n^2 \, dx  + C_5^2(\omega) \int_{|x| > 1} |x|^2 (f(v_n))^2 \, dx \\
\leq \int_{|x| \leq 1} [(f(v_n))^2 + 2 (f(v_n))^4] \, dx  + C_5^2(\omega) \int_{|x| > 1} |x|^2 (f(v_n))^2 \, dx \leq C_6(\omega) \,.
\end{multline*}
In particular, $(v_n)$ is uniformly bounded in $H$ and standard arguments show that a weak limit $v$ is a global minimizer of $\mathcal{A}_\omega$.

\smallbreak

Therefore, we conclude that $v \in H$ is such that $v \geq 0$ and realizes $m:=\inf_{H} \mathcal{A_{\omega}}$. Then, by the maximum principle, we infer that $v \in A_{\omega}$, as desired.
\end{proof}

\begin{remark}
All results proved in this section can be extended to
\begin{equation}\label{eq:generalequation}
  -\Delta u-|u|^{\alpha-2}u\Delta |u|^\alpha+ |x|^\gamma u +  |u|^{p-1} u =\omega u,\qquad  \text{in $\Omega$},
\end{equation}
with $\alpha>1$, $p>1$, and $\gamma>0$, a generalization considered in \cite{AdachiWatanabe} (without the trapping potential, and in the defocusing case). The only difference in the proof of Theorems {\rm{\ref{thmQLNLS_bounded}}} and {\rm{\ref{thmQLNLS_unbounded}}} would be the definition of  the energy functional  
\[
\mathcal{E}(v)  =\frac{1}{2}\int_{\Omega} |\nabla v|^2 \, dx- \frac{1}{2} \int_\Omega (f(v))^2 \, dx+ \frac{1}{2}\int_{{\mathbb R}^N} |x|^\gamma (f(v))^2
  + \frac{1}{p+1}\,dx\int_{{\mathbb R}^N} |f(v)|^{p+1}\, dx,
\] 
and the action functional 
\[
\mathcal{A}_\omega(v)=\mathcal{E}(v)-  \omega \int_\Omega (f(v))^2\, dx,
\]
where $f$ is the odd function which solves
\[
f'(t)=\frac{1}{\sqrt{1+\alpha f^{2\alpha -2}(t)}} \text{ in } (0,\infty),\quad f(0)=0.
\]

\end{remark}

\begin{remark}
We end this section by observing that Selvitella in \cite{selvitella} proved existence results for some general problems that include \eqref{eq:generalequation}, proving also \emph{nonoptimal} decay estimates.
\end{remark}

\subsection{Defocusing Gross-Pitaevskii system}\label{subsecGP}
Consider the following system with $k$ equations:
\begin{equation}\label{eq:BEC_bounded}
-\Delta u_i +V(x) u_i +u_i \sum_{j=1}^k \beta_{ij} u_j^2=\omega_i u_i \quad \text{ in } \Omega,\qquad i=1,\ldots, k,
\end{equation}
where $V(x)=|x|^2$ if $\Omega={\mathbb R}^N$, or $V\in L^\infty(\Omega)$ if $\Omega\subset {\mathbb R}^N$ is a bounded regular domain, and $N\geq 1$. Such system arises when looking for standing wave solutions $\phi_i(t,x)=e^{\imath \omega_i t}u_i(x)$ of the following Gross-Pitaevskii system:
\begin{equation}\label{eq:BEC_complex}
\imath \partial_t \phi_i-\Delta \phi_i +\phi_i \sum_{j=1}^k \beta_{ij} |\phi_j|^2=0\,.
\end{equation}
We will assume that
\begin{equation}\label{eq:assumptions_B}
B =  (\beta_{ij})_{ij}\ \text{symmetric, and either positive semidefinite or positive definite.}
\end{equation}
The symmetry assumption on $B$ makes the problem variational. The positive semidefiniteness implies that $\beta_{ii} \geq 0$ for $i=1, \ldots, k$. In the case $\beta_{ii}> 0$, for every $i=1, \ldots,k$, it is usually said that the self-interacting parameters are of defocusing type.

\begin{remark}
Assume $\beta_{11}, \ldots, \beta_{kk}> 0$. With $k=2$, $(\beta_{ij})_{ij}$ being positive semidefinite is equivalent to $\beta_{12}^2\leq \beta_{11}\beta_{22}$. For a general $k$, this assumption is fulfilled for instance when the off-diagonal terms $\beta_{ij}$ ($i\neq j$) are small with respect to the diagonal ones $\beta_{ii}$.
\end{remark}

 These systems appear as a model in the physical phenomenon of Bose-Einstein Condensation or in Nonlinear Optics (see for instance \cite{ChangLinLinLin,Sirakov,Timmermans} and references therein for an easy-to-follow physical description). From a mathematical point of view, there has been an intense activity in the last ten years, both regarding existence results (check for example the introduction of \cite{SoaveTavares}, or \cite{BartschJeanjeanSoave,NorisTavaresVerzini} and their references) as well as the regularity and asymptotic study of solutions as the competition increases, namely $\beta_{ij}\to +\infty$ for $i\neq j$ (see the recent survey \cite{SoaveTavaresTerraciniZilio} for an overview on this topic). One of the interesting features of these systems is that they admit solutions $(u_1,\ldots, u_k)$ with trivial components, that is $u_i\equiv 0$ for some $i$.

\smallbreak

In order to consider solutions of \eqref{eq:BEC_bounded}, there are two (nonequivalent) points of view. The first is to consider $\omega=(\omega_1,\ldots, \omega_k)\in {\mathbb R}^k$ fixed, and look for solutions as critical points of the action functional
\begin{align}\label{eq:ActionSystems}
\mathcal{A}_\omega(u_1,\ldots, u_k) &=\frac{1}{2}\sum_{i=1}^k \int_\Omega (|\nabla u_i|^2 +(V(x)-\omega_i) u_i^2)\, dx+ \sum_{i=1}^k \frac{\beta_{ii}}{4}\int_\Omega u_i^4 \, dx+\sum_{i<j} \frac{\beta_{ij}}{2}\int_\Omega u_i^2 u_j^2\, dx\,,\\
&=\frac{1}{2}\sum_{i=1}^k \int_{\Omega} (|\nabla u_i|^2+(V(x)-\omega_i)u_i^2)\, dx+\frac{1}{4} \int_{\Omega} [u_1^2 \ldots u_k^2] B [u_1^2 \ldots u_k^2]^T\, dx
\end{align}
defined in $H$, where $H=\{u\in (H^1({\mathbb R}^N)\cap L^4(\mathbb{R}^N))^k;\ \int_{{\mathbb R}^N} u_i^2 |x|^2\, dx<\infty\ \textrm{ for all } i \}$ and 
\[
\|u\|_H:=\sum_{i=1}^k\left(\int_{\mathbb{R}^N} (|\nabla u|^2+|x|^2 u_i^2) \,dx \right)^{1/2}+\sum_{i=1}^k \left(\int_{\mathbb{R}^N} u_i^4\, dx\right)^{1/4}
\] if $\Omega={\mathbb R}^N$, or $H  =  (H^1_0(\Omega)\cap L^4(\Omega))^k$ and
\[
\|u\|_H:=\sum_{i=1}^k\left(\int_{\Omega} |\nabla u|^2\, dx\right)^{1/2}+\sum_{i=1}^k \left(\int_{\Omega} u_i^4\, dx\right)^{1/4}
\]
if $\Omega$ is bounded. 

The second point of view consists of fixing the $L^2$--norm (i.e., the mass) rather that the $\omega_i$'s: 
\[
\mathcal{M}=\{(u_1,\ldots, u_k)\in H; \ \int_\Omega u_i^2=1,\ \textrm{ for all } i=1,\ldots, k\}
\]
and to look for solutions of \eqref{eq:BEC_bounded} as critical points of the energy functional
\begin{align}\label{eq:EnergySystems}
\mathcal{E}(u_1,\ldots, u_k)&=\frac{1}{2}\sum_{i=1}^k \int_\Omega (|\nabla u_i|^2+V(x) u_i^2) + \sum_{i=1}^k \frac{\beta_{ii}}{4}\int_\Omega u_i^4 +\sum_{i<j} \frac{\beta_{ij}}{2}\int_\Omega u_i^2 u_j^2\\
				&=\frac{1}{2}\sum_{i=1}^k \int_{\Omega} (|\nabla u_i|^2+V(x)u_i^2)+\frac{1}{4} \int_{\Omega} [u_1^2 \ldots u_k^2] B [u_1^2 \ldots u_k^2]^T
\end{align}
constrained to $\mathcal{M}$. Among these critical points, minimizers of $\mathcal{E}|_\mathcal{M}$ are called \emph{ground state solutions}. Within this second framework, the parameters $\omega_i$ in \eqref{eq:BEC_bounded} are not fixed a priori, but  appear as Lagrange multipliers. These solutions are physically relevant since both the energy and the mass are, at least formally, conserved along trajectories $t\mapsto (\phi_1(t),\ldots, \phi_k(t))$ of the solutions of system \eqref{eq:BEC_complex}. 

In what follows, we will say that $(u_1,\ldots, u_k)$ is \emph{positive} if $u_i>0$ for every $i=1,\ldots, k$.  Using the first point of view leads to the set
\begin{equation}\label{eq:A_action2}
A_{\omega}=\left\{(u_1,\ldots, u_k)\in H;\ u_i>0\ \textrm{ for all } i,\ \mathcal{A}_{\mathbf{\omega}}'(u_1,\ldots, u_k)=0 \right\},
\end{equation}
while the second leads to
\begin{equation}\label{eq:A_energy2}
A_{\mathcal E}=\{(u_1,\ldots,u_k) \in \mathcal{M}; \, u_i>0 \ \textrm{ for all } i,\  \mathcal{E}'|_{\mathcal M}(u_1,\ldots, u_k)=0\}.
\end{equation}
These sets in principle do not coincide, since two critical points of $\mathcal{E}|_\mathcal{M}$, or even two ground states, may have different associated Lagrange multipliers. 

Observe that if the set $A_{\omega}$ is not empty then necessarily $\omega_i>\lambda_V$ for at least one $i \in \{1, \ldots, k\}$, where $\lambda_V$ stands for the first eigenvalue of $(-\Delta+V(x)I, H$). The equivalence is not clear since the functional $\mathcal{A}_\omega$ admits, in general, critical points with trivial components. 

On the other hand, the set ${A}_{\mathcal{E}}$ is non empty if $B$ is positive semidefinite, containing necessarily a ground state solution, that is, a minimizer of $\mathcal{E}|_\mathcal{M}$. In fact, the level $\min_\mathcal{M} \mathcal{E}$ is clearly achieved by $(u_i, \ldots, u_k)$, a critical point of $\mathcal{E}|_\mathcal{M}$ with all positive components (by eventually taking $|u_i|$ and using the maximum principle). Observe that $H\subset\subset L^2(\Omega)$ in both cases $\Omega=\mathbb{R}^N$ or $\Omega$ bounded.

 \smallbreak

Let us first discuss the case $\Omega$ bounded.  Using Theorem \ref{th:abstracttheorem} and Corollary \ref{cor:abstractcor}, our contribution is the following new result.

\begin{theorem}\label{thm:BEC_kequations} Let $\Omega$ be a bounded regular domain and $B=(\beta_{ij})_{ij}$ a symmetric matrix.
\begin{enumerate}[i)]
\item Assume that $B$ is positive definite. Then, for $\omega_1,\ldots, \omega_k \in {\mathbb R}$  fixed, the functional $\mathcal{A}_\omega$ has at most one positive critical point, that is, system \eqref{eq:BEC_bounded} has at most a positive solution.
\item Assume that $B$ is positive semidefinite. Then there exists exactly one positive critical point of $\mathcal{E}|_\mathcal{M}$. In particular, the ground state is unique, up to sign.
\end{enumerate}
\end{theorem}

\begin{remark}\label{rem:nonuniqueness_Bpositivedefinite}
As it will be illustrated in the proof of Theorem {\rm{\ref{thm:BEC_kequations}}}, if $B$ is not positive definite then in item i) the only thing we can conclude in general is that either $A_\omega=\emptyset$ or there exists $u=(u_1,\ldots, u_k)\in A_\omega$ such that
\begin{equation}\label{eq:inclusionsystems}
A_\omega\subset \{(\alpha_1 u_1,\ldots, \alpha_k u_k);\ \alpha_i>0\ \text{ for every } i\}.
\end{equation}
Observe that the strict convexity of Lemma {\rm{\ref{lemma:strictconvexity}} is essential in this case.} There are particular cases when $B$ is positive semidefinite, $\det(B)=0$, and $A_\omega$ is not a singleton: if for instance $\Omega$ is bounded with $\lambda_1(\Omega)<1$, $V(x)\equiv 0$, $k=2$, $\omega_1=\omega_2=1$ and $\beta_{ij}=1$ for every $i,j=1,2$, then (taking also into account \eqref{eq:inclusionsystems} and Example \ref{example-Allen-Cahn}):
\[
A_{0}=\left\{\left(\alpha_1 u, \alpha_2 u\right);\ -\Delta u= u-u^3,\ u>0 \text{ in } \Omega, \ u=0 \text{ on } \partial \Omega,\ \ \alpha_1^2+\alpha_2^2=1, \ \alpha_1,\alpha_2>0 \right\} .
\]
\end{remark}

\smallbreak
The result of Theorem \ref{thm:BEC_kequations} i) was proved in \cite[Theorem 4.1]{AlamaBronsardMironescu} by using different techniques. 
The proof there consists of using the identity obtained when one multiplies the $i$-th equation of \eqref{eq:BEC_bounded} by $\frac{1}{2}u_i((\frac{u_i}{v_i})^2-1)$, where $u$ and $v$ are two positive solutions. 

\begin{remark}
If $\omega_i>\lambda_{V}$ for every $i$, it is straightforward to prove that $A_\omega\neq \emptyset$, and thus this set is a singleton if $B$ is positive definite. In fact, the level
$\min_H \mathcal{A}_\omega$ is clearly achieved, and if it were achieved by a vector with some trivial components, say for e.g. $(0,u_2,\ldots, u_k)$, then by taking $\varphi_1$ an eigenfunction associated with $\lambda_V$ and using the fact that $\omega_1>\lambda_V$ it is easy to show that
\[
\mathcal{A}_\omega (\varepsilon \varphi_1,u_2,\ldots, u_k)<\mathcal{A}_\omega (0,u_2,\ldots, u_k),
\]
a contradiction.
\end{remark}

As an immediate consequence of the second statement of Theorem \ref{thm:BEC_kequations}, we have the following classification result.

\begin{corollary}\label{cor:samebetas}
 Assume $\Omega$ is a bounded regular domain, $V(x)\equiv 0$,  $B=(\beta_{ij})_{ij}$ is a symmetric positive semidefinite matrix, and that there exists $\beta$ such that
 \[ \sum_{j=1}^{k} \beta_{ij} = \beta \quad \textrm{ for all }  i=1, \ldots, k.
 \]
 Then the unique positive critical point of $\mathcal{E}|_{\mathcal{M}}$ is $u = (U, \ldots, U)$ where $U$ is the unique positive critical point of $\mathcal{\tilde E}|_S$, where  
 \[
 \mathcal{\tilde E} (w)  = \frac{1}{2} \int_{\Omega} |\nabla w|^2 dx - \frac{\beta}{4} \int_{\Omega} w^4, \quad w \in H^1_0(\Omega)
 \]
 and $S = \{ w \in H^1_0(\Omega); \int_{\Omega} w^2dx =1\}.$
\end{corollary}

In the case $\Omega={\mathbb R}^N$ (with trapping potential), our results are the following.

\begin{theorem}\label{thm:GP_R^N}
Assume that $B=(\beta_{ij})_{ij}$ is a symmetric matrix, $\Omega={\mathbb R}^N$, and $V(x)=|x|^2$. Then,
\begin{enumerate}[i)]
\item[i)] Assume that $B$ is positive definite. Given $\omega_1,\ldots, \omega_k\in \mathbb{R}$, the system \eqref{eq:BEC_bounded} admits at most one positive solution.
\item[ii)] Assume that $B$ is positive semidefinite. There exists a unique positive critical point of $\mathcal{E}|_{\mathcal{M}}$. In particular, the ground state is unique, up to sign.
\end{enumerate}
\end{theorem}

The first result is a generalization of \cite[Theorem 1.3-(1)]{AftalionNorisSourdis} for systems with an arbitrary number of equations. Moreover, our proof seems simpler. The second results is a generalization of \cite[Theorem 1.3-(2)]{AftalionNorisSourdis} (which holds for two equations).

\smallbreak

In the proof of the previous theorem, a key step is the \emph{sharp decay at infinity} of the solutions of \eqref{eq:BEC_bounded}.

\begin{proposition}\label{prop:sharpdecay}
Let $(u_1,\ldots, u_k)\in H$ be a positive solution of 
\[
-\Delta u_i +|x|^2 u_i +u_i \sum_{j=1}^k \beta_{ij} u_j^2=\omega_i u_i \qquad \text{ in } \mathbb{R}^N.
\] Then, for every $i=1,\ldots, k$,
\[
u_i(x)\sim |x|^{\frac{\omega_i-N}{2}}e^{-\frac{1}{2}|x|^2} \qquad \text{ as } |x|\to \infty.
\]
\end{proposition}

This result will replace Hopf's Lemma (not available when working in $\mathbb{R}^N$) and imply a global comparison principle between positive solutions; this will be used in the proof of item i). Since the decay depends on the Lagrange multipliers $\omega_i$, in order to prove ii) we will not have at our disposal a global comparison between solutions, and instead of using Corollary \ref{cor:abstractcor} we will combine the sharp decay information with the strategy previously used in \cite{AftalionNorisSourdis}.

\smallbreak

Let us proceed for the proofs of the results.

\begin{proof}[Proof of Theorem {\rm{\ref{thm:BEC_kequations}}}]

Let $A$ be either \eqref{eq:A_action2} (for item (i)) or \eqref{eq:A_energy2} (for item (ii)). Given $u = (u_1, \ldots, u_k), v = (v_1, \ldots, v_k)\in A$ with $u\neq v$, take 
\[
\gamma(t)=(\gamma_1(t),\ldots, \gamma_k(t)),\qquad \text{ with } \gamma_i(t)=\sqrt{(1-t) u_i^2+t v_i^2}.
\]
By the standard Hopf Lemma, we have that the comparison condition \eqref{eq:hipderivative2} in Section \ref{sec:aux-res} is satisfied for each pair $(u_i,v_i)$, and Corollary \ref{corollary_of_lemma:gpg} yields that $\gamma$ is Lipschitz continuous at 0. The assumption that $(\beta_{ij})$ is positive semi-definite ensures that the quadratic form
\[
{\mathbb R}^k\to {\mathbb R},\qquad z\mapsto z B z^{T} =\sum_{i} \beta_{ii} z_i^2 + 2\sum_{i<j} \beta_{ij}z_i z_j
\]
is convex and, using Lemma \eqref{lemma:strictconvexity}, we infer that
\[
t \mapsto  \mathcal{A}_\omega(\gamma(t)) \quad \text{and} \quad t\mapsto  \mathcal{E}(\gamma(t)) \quad \text{are strictly convex on [0,1]}.
\]
Indeed, it is clear that $t \mapsto  \mathcal{A}_\omega(\gamma(t))$ is strictly convex if $B$ is positive definite.  In the second case, with $u, v \in \mathcal{M}$ and since $u \neq v$, there exists $i \in \{1, \ldots, k\}$ such that $u_i$ and $v_i$ are linearly independent and, in this case, the strict convexity of $t\mapsto  \mathcal{E}(\gamma(t))$ follows from the strict convexity of the gradient term given by Lemma \ref{lemma:strictconvexity}. Therefore, in the case of i) the conclusion follows by Theorem \ref{th:abstracttheorem}, and we infer that $A_{\mathcal E}$ is empty or a singleton by Corollary \ref{cor:abstractcor}. As above, recall that $A_\omega\neq\emptyset$ by standard compactness arguments.
\end{proof}

\begin{proof}[Proof of Proposition {\rm{\ref{prop:sharpdecay}}}]
The proof closely follows the proof of Theorem \ref{prop:sharpdecay_qlS} and we only highlight the differences. 
First of all, observe that $u_i, |\nabla u_i|\in L^\infty(\mathbb{R}^N)$, and $u_i\to 0$ as $|x|\to \infty$. In fact, note that the potential $|x|^2$ is bounded on the ball $B_{2\sqrt{\omega_i}}(0)$, being bounded away from zero outside, and the Moser iteration
scheme applies. Next, as in
Step 2. in the proof of Theorem \ref{prop:sharpdecay_qlS} one has 
that for each $i$
\[
-\Delta u_i+(|x|^2-\omega_i-\kappa) u_i\leq 0.
\]
and analogously one obtains 
\[
h(x)\leq C e^{-\frac{|x|^2}{4}} \qquad |x| \geq R.
\]
for some $C, R > 0$. Then,   
\[
\left| \sum_{j=1}^k \beta_{ij} u_j^2\right|\leq \alpha e^{-\frac{|x|^2}{2}} \,, 
\]
and therefore 
\[
-\Delta u_i+(|x|^2-\omega_i-\alpha e^{-\frac{|x|^2}{2}})u_i\leq 0, \quad -\Delta u_i+(|x|^2-\omega_i+\alpha e^{-\frac{|x|^2}{2}})u_i\geq 0,
\]
in ${\mathbb R}^N\setminus \overline{B_R(0)}$. The proof now follows as
in the proof of Theorem \ref{prop:sharpdecay_qlS} with obvious modifications. 
\end{proof}

\smallbreak

\begin{proof}[Proof of Theorem {\rm{\ref{thm:GP_R^N}}}] i) We can repeat the proof of Theorem \ref{thm:BEC_kequations}-i) almost word by word, with the exception that now, instead of using Hopf's lemma, we get the comparison condition \eqref{eq:hipderivative2} in Section \ref{sec:aux-res} for any two possible positive solutions of \eqref{eq:BEC_bounded} from the sharp decay of Proposition \ref{prop:sharpdecay}. 
\smallbreak
\noindent ii) Our main results Theorem \ref{th:abstracttheorem} and Corollary \ref{cor:abstractcor}
cannot be directly applied in this case, since two critical
points $u$ and $v$ can have different Lagrange multipliers $\omega$, 
and therefore different decay at infinity. In such a situation it is 
not clear how to choose a path $\gamma$ which is Lipschitz continuous
at the endpoints. 

Instead, our strategy will basically follow the ideas of  \cite[Theorem 1.3-(2)]{AftalionNorisSourdis} (see also \cite[Theorem 4.1]{AlamaBronsardMironescu}), and we refer to it for the exact computations and more details. Our main contribution here is 
an extension to $k$ equations and weaker assumptions on $B$.
The key step is an use of the sharp estimates of Proposition \ref{prop:sharpdecay}. 

For any two positive critical points  $u,v$ we claim that 
\begin{equation}\label{eq:energy_identity}
\mathcal{E}(u_1,\ldots, u_k)=\mathcal{E}(v_1,\ldots, v_k)+F(w_1,\ldots, w_k)
\end{equation}
where $w_i := u_i/v_i$ and 
\begin{align*}
F(w_1,\ldots, w_k) &=\sum_{i=1}^k \int_{{\mathbb R}^N} \frac{|\nabla w_i|^2v_i^2}{2} + \frac{1}{4} \int_{{\mathbb R}^N} [v_1^2 (w_1^2-1) \ldots v_k^2 (w_k^2-1)] B [v_1^2 (w_1^2-1) \ldots v_k^2 (w_k^2-1)]^T \,.
\end{align*}
\smallbreak 
Assume that $(u_i)$ and $(v_i)$ satisfy \eqref{eq:BEC_bounded} respectively with Lagrange multipliers $(\omega_i)$ and $(\mu_i)$. 
By using integration by parts, we have, for every $R>0$ fixed,
\[
\int_{B_R(0)} -\Delta v_i v_i (w_i^2-1)+ \int_{\partial B_R(0)} \partial_\nu v_i \left(\frac{u_i^2}{v_i}-v_i\right)=\int_{B_R(0)} \left( |\nabla v_i|^2(w_i^2-1)+2v_i w_i \nabla v_i \cdot \nabla w_i\right) \,.
\]
Observe that, since $\nabla v_i\in L^\infty$ and 
\[
u_i(x)\sim |x|^{\frac{\omega_i-N}{2}}e^{-\frac{1}{2}|x|^2},\qquad v_i(x)\sim |x|^{\frac{\mu_i-N}{2}}e^{-\frac{1}{2}|x|^2},
\]
then
\[
\left| \int_{\partial B_R(0)} \partial_\nu v_i \left(\frac{u_i^2}{v_i}-v_i\right)\right| \leq C R^{N-1} R^{\omega_i-N -\frac{\mu_i-N}{2}}e^{-\frac{R^2}{2}} +C R^{N-1}  R^{\frac{\mu_i-N}{2}}e^{-\frac{1}{2}R^2}
\to 0 \quad \textrm{ as } R \to \infty \,.
\]
Consequently, since $u$ solves \eqref{eq:BEC_bounded}
\[
\int_{{\mathbb R}^N} (|\nabla v_i|^2 (w_i^2-1) +2 v_i w_i \nabla v_i \cdot \nabla w_i) =-\int_{{\mathbb R}^N} (|x|^2v_i^2 +\beta_{ii}v_i^4 +v_i^2 \sum_{j\neq i} \beta_{ij} v_j^2) (w_i^2-1),
\]
where we used that  $u_i$ and $v_i$ have the same mass, that is, $(u_1,\ldots, u_k), (v_1,\ldots, v_k)\in \mathcal{M}$, and therefore
\[
 \mu_i\int_{{\mathbb R}^N} v_i^2 (w_i^2-1)=\mu_i \int_{{\mathbb R}^N} (u_i^2-v_i^2)=0.
\]
Now, by using this identity in $\mathcal{E}(u_1,\ldots, u_k)=\mathcal{E}(v_1 w_1,\ldots, v_k w_k)$, we get \eqref{eq:energy_identity}.

Since $B$ is positive semidefinite, $F(w_1,\ldots, w_k)\geq 0$ and 
\eqref{eq:energy_identity} yields $\mathcal{E}(u_1\ldots, u_k) \geq \mathcal{E}(v_1\ldots, v_k)$. By interchanging $u$ and $v$, we have $\mathcal{E}(u_1\ldots, u_k)=\mathcal{E}(v_1\ldots, v_k)$, and thus $F(w_1,\ldots, w_k)=0$.
This, in turn, gives $w_i^2\equiv \alpha_i$, for some constant $\alpha_i$. Actually $\alpha_i=1$, since $\int_{\mathbb{R}^N}u_i^2\, dx=\int_{\mathbb{R}^N} v_i^2\, dx=1$.

Once again $A_\omega\neq\emptyset$ by standard compactness arguments.
\end{proof}

 \section*{Acknowledgements}
 D. Bonheure is supported by INRIA - Team MEPHYSTO, 
MIS F.4508.14 (FNRS), PDR T.1110.14F (FNRS) 
\& ARC AUWB-2012-12/17-ULB1- IAPAS. 

E. Moreira dos Santos is partially supported by CNPq  projects 309291/2012-7, 490250/2013-0 and 307358/2015-1 and FAPESP projects 2014/03805-2 and 2015/17096-6. 

H. Tavares is supported by Funda\c c\~ao para a Ci\^encia e Tecnologia through the program \emph{Investigador FCT} and the project PEst-OE/EEI/LA0009/2013, and by the project ERC Advanced Grant  2013 n. 339958 ``Complex Patterns for Strongly Interacting Dynamical Systems - COMPAT''.

D. Bonheure \& E. Moreira dos Santos are partially supported by a bilateral agreement FNRS/CNPq. 

D. Bonheure, J. F\"{o}ldes, and A. Salda\~{n}a are supported by MIS F.4508.14 (FNRS).

Part of this work was done while the third author was visiting the Dipartimento di Matematica of the Universit\`a di Roma {\it{Sapienza}}, whose hospitality he gratefully acknowledges.

\end{document}